\documentclass[12pt]{amsart}
\usepackage{fontenc}
\usepackage[english]{babel}
\usepackage[swapnames,normal]{frontespizio}

\usepackage[font=small,format=default,indention=1em,textfont=it,up]{caption}
\usepackage{enumitem}
\usepackage{subfig}
\usepackage{url}									
\usepackage{hyperref}
\theoremstyle{plain}
\newtheorem{theorem}{Theorem}[section]
\newtheorem{lemma}[theorem]{Lemma}
\newtheorem{proposition}[theorem]{Proposition}
\newtheorem{corollary}[theorem]{Corollary}

\theoremstyle{remark}
\newtheorem{example}[theorem]{Example}
\theoremstyle{remark}
\newtheorem{remark}[theorem]{Remark}
\theoremstyle{definition}
\newtheorem{definition}[theorem]{Definition}
\numberwithin{equation}{section}

\def\H{\mathcal{O}}

\def\I{\mathcal{P}}
\newcommand{\N}{\mathbb{N}}
\newcommand{\0}{\mathbb{N}_{0}}
\renewcommand{\k}{\Xi}
\newcommand{\h}{\Theta}
\newcommand{\K}{K}

\begin{document}
\title[Alternating and Variable Controls]
{Alternating and Variable Controls\\ for the Wave Equation}

\author{Antonio Agresti}
\author{Daniele Andreucci}
\author{Paola Loreti}

\thanks{The original publication is available at \url{www.esaim-cocv.org}. The copyright of this article belongs to ESAIM-COCV}
\thanks{
The second author is member of Italian G.N.F.M.-I.N.d.A.M.}

\address{Dipartimento di Scienze di Base e Applicate per l'Ingegneria\\
Sapienza Universit\`a di Roma\\
via A.Scarpa 16, 00161 Roma, Italy}
\curraddr{Dipartimento di Matematica Guido Castelnuovo\\
Sapienza Universit\`a di Roma\\
P.le A. Moro 2, 00100 Roma, Italy}
\address{Dipartimento di Scienze di Base e Applicate per l'Ingegneria\\
Sapienza Universit\`a di Roma\\
via A.Scarpa 16, 00161 Roma, Italy}
\address{Dipartimento di Scienze di Base e Applicate per l'Ingegneria\\
Sapienza Universit\`a di Roma\\
via A.Scarpa 16, 00161 Roma, Italy}

\begin{abstract}
The present article discusses the \textit{exact observability} of the wave equation when the observation subset of the boundary is variable in time. In the one-dimensional case, we prove an equivalent condition for the exact observability, which takes into account only the location in time of the observation. To this end we use Fourier series. Then we investigate the two specific cases of single exchange of the control position, and of exchange at a constant rate.
In the multi-dimensional case, we analyse sufficient conditions for the exact observability relying on the multiplier method.
In the last section, the multi-dimensional results are applied to specific settings and some connections between the one and multi-dimensional case are discussed; furthermore some open problems are presented.
\end{abstract}

\maketitle

\section{Introduction and formulation of the main results.}
\label{s:introduction}
In this work we study exact observability for the wave equation with Dirichlet boundary condition, i.e.
\begin{align}
\label{omo1}
u_{tt} - \Delta u&=0\,, &\qquad& (x,t)\in \Omega\times(0,T)\,,\\
u(x,t)&=0\,,&\qquad& (x,t)\in\partial \Omega \times (0,T)\,,\\
u(x,0)&=u_0(x)\,,&\qquad& x\in \Omega\,,\\
\label{omo2}
u_t(x,0)&=u_1(x)\,,&\qquad& x\in \Omega\,;
\end{align}
when the subset of observation is not (in general) of the form $\Gamma\times(0,T)\subset \partial \Omega \times (0,T)$ for a fixed $\Gamma\subset \partial\Omega$.\\
Here $\Delta$ is the Laplace operator and $\Omega$ is a domain (i.e., an open, connected and bounded set) in $\mathbb{R}^d$, which we consider of class $C^2$ \textit{or} convex (for more information see \cite{Grisvard}).

For the reader's convenience, we have gathered in Section \ref{s:prelim} some known results, and some extensions we need, on basic existence and regularity theory for the problem above.
In the following the trace of the normal derivative of weak solutions will be understood in the sense of Theorem \ref{t:hidden}.

Since the treatment of the observability problem (see below for a definition or \cite{Fourier}) is different in the one dimensional and in the multi-dimensional case, due to the different topology of the subset of observation, we split the discussion into two subsections.
\subsection{One-dimensional case ($d=1$)}
\label{s:alternating}
In this case, w.l.o.g. we will consider $\Omega=(0,\pi)$. We are interested in studying the controllability property of the one dimensional wave equation with Dirichlet boundary condition (e.g., by a control at $x=0$). It is known (see for instance \cite{multiplier} or \cite{Lions}) that the exact controllability is equivalent to the observability inequalities such as
\begin{equation}
\label{observability classica}
\int_0^T |u_x (0,t)|^2\,dt \asymp  E_0\,.
\end{equation}
Here and in the rest of the paper the symbol $A \asymp B$  means that $c^{-1}A<B<cA$ for some positive real $c$; where $u$ solves (\ref{omo1})-(\ref{omo2}) and $E_0$ is the energy of the initial data $(u_0,u_1)\in H^1_0(0,\pi)\times L^2(0,\pi)$; see (\ref{energy}) below. It is also clear that the observability inequality (\ref{observability classical}) remains valid if $u_x(0,t)$ is replaced by $u_x(\pi,t)$, which provides the controllability for a system controlled at $x=\pi$.\\
It is well known that (\ref{observability classical}) holds for any $T\geq 2\pi$, as a consequence of the Parseval equality, for a proof see \cite{Fourier}. Moreover, if the control is actuated on both the endpoints, the optimal control time $T$ reduces to $\pi$.\\
Although both the inequalities require an investigation, for the study we are going to explain the  so called {\em{direct observability inequality}}  follows by  hidden regularity property (see Theorem \ref{t:hidden}):

For any $T>0$ and $(u_0,u_1)\in H^1_0(\Omega)\times L^2(\Omega)$, the map
\begin{align*}
H^1_0(\Omega)\times L^2(\Omega) &\rightarrow L^2(\partial \Omega\times (0,T))\,,\\
(u_0,u_1)&\mapsto \partial_{\nu} u\,,
\end{align*}
is well defined, continuous and coincides with the classical trace (in the sense of Sobolev spaces) if $(u_0,u_1)\in (H^2(\Omega)\cap H^1_0(\Omega)) \times H^1_0(\Omega)$.

In the paper, we concentrate our analysis on the { \em{inverse observability inequality}} 
 \begin{equation}
\label{observability classical}
\int_0^T |u_x (0,t)|^2\,dt \geq c  E_0\,.
\end{equation}
For brevity, from now the inequality \eqref{observability classical} will be 
called \textit{observability inequality}.\\
\\
In this work  we want to generalize the situation of having the observability set as $\{0\}\times (0,T)$ or $\{\pi\}\times (0,T)$; that is we want to alternate the spatial controllability point still maintaining the observability inequality true, for an optimal control time. We are not aware of specific contributions in this sense with the exception the recent paper \cite{Variable}, which mostly treats bulk control, but touches the variable boundary control; however the methods of \cite{Variable} are completely different from ours. The motivation for this generalization came from typical engineering problems (see \cite{Carcaterra2}) in which the control sets have to change in time, so that, also the observability set has to change in time.  

Consider an arbitrary family of consecutive disjoint open intervals $\{I_k\,:\,k\in \mathcal{K}\}$ of $(0,\infty)$ with $\mathcal{K}\subset \0$ (possibly finite or infinite); we define
the \textit{observability sets} 
\begin{equation}
\label{def:F}
\mathcal{F}:=\{\{\lambda_k\}\times I_k\,:\,k\in\mathcal{K}\}\,,
\end{equation}
where $\lambda_k=0$ if $k$ is even, otherwise $\lambda_k=\pi$. Without loss of generality we can suppose $\mathcal{K}=\{0,\dots,N\}$ if $\mathcal{K}$ is finite or $\mathcal{K}=\0$ otherwise; furthermore we set $(t_{k-1},t_k):= I_k$ and $t_{-1}=0$.\\
We want to investigate whether the following property is satisfied:
\begin{definition}[$\mathcal{F}$-observability]
\label{def:observability1}
We say that the system (\ref{omo1})-(\ref{omo2}) is $\mathcal{F}$-observable in time $T>0$ if the inverse observability inequality holds true, that is
\begin{equation}
\label{observability0}
\sum_{k=0}^n \int_{I_k\cap (0,T)} |u_x(\lambda_k,t)|^2\, dt \geq c\,E_0\,,
\end{equation} 
for some $n\in\N$ and $c>0$ independent of $(u_0,u_1)\in H^1_0(0,\pi)\times L^2(0,\pi)$.
As above, $u$ is the weak solution to $(\ref{omo1})-(\ref{omo2})$ with initial data $(u_0,u_1)$.\\
We shall say that the system (\ref{omo1})-(\ref{omo2}) is $\mathcal{F}$-observable, if the system is $\mathcal{F}$-observable for some time $T>0$.
\end{definition}

In applications we are interested in knowing the least time $T$ such that the $\mathcal{F}$-observability property holds; we make this more precise in the following definition.
\begin{definition}[Optimal control time]
\label{defTopt}
Suppose that the system (\ref{omo1})-(\ref{omo2}) is $\mathcal{F}$-observable in time $T$, but it is not $\mathcal{F}$-observable in any time $T'<T$. Then $T_{opt}:=T$ will be called the \textit{optimal control time}.
\end{definition}
In the following the family of intervals $\{I_k:k\in \mathcal{K}\}$ is given. In order to state our result on the observability property, we have to introduce some notations. Let $r$ be the quotient map
\begin{equation*}
r: \mathbb{R} \rightarrow \mathbb{R}/(2\pi \mathbb{Z}) \approx S^1\,.
\end{equation*}
Lastly we define
\begin{align}
\label{condition}
\tilde{I}_k=\{ r(t)\,:\,t\in I_k-\lambda_k\}\,,
 \qquad k\in \0\,,
\end{align}
where as above, $\lambda_k=0$ if $k\in 2\0$, $\lambda_k=\pi$ otherwise. We are ready to state our main Theorem.
\begin{theorem}
\label{general theorem}
Let $n\in \0$. The system (\ref{omo1})-(\ref{omo2}) is $\mathcal{F}$-observable in time $T=\sup(I_n)$ if and only if
\begin{equation}
\label{general condition}
\Big| S^1 \setminus \bigcup_{k=0}^n \tilde{I}_k \Big| =0\,.
\end{equation}
Furthermore, let $h$ be the least index $n$ in $\mathcal{K}$ for which the condition (\ref{general condition}) holds, 
then the optimal control time is the infimum over the values $t$ for which 
\begin{equation*}
\Big| S^1 \setminus \left(\bigcup_{k=0}^{h-1} \tilde{I}_k\cup r((t_{h-1}-\lambda_h,t-\lambda_h) )\right) \Big| =0\,.
\end{equation*}
\end{theorem}
If $\mathcal{K}=\{0\}$, then the family of subsets considered in Theorem \ref{general theorem} contains a single interval. In this case the content of Theorem \ref{general theorem} reduces to the one point observation result discussed at the beginning of the subsection. Again, for a complete proof in this special case we refer to \cite{Fourier}. Finally, we note that in this case, according to Definition \ref{defTopt}, the optimal control time is $T_{opt}=2\pi$.\\
Interestingly, one has the following refinement of Theorem \ref{general theorem}.
\begin{corollary}
\label{corollaryJ} 
If the condition (\ref{general condition}) holds, then we can choose a family $\{\tilde{J}_k :k=0,\dots,n\}$ of disjoint subsets of $S^1$ (possibly empty), such that $\tilde{J}_k \subset \tilde{I}_k$ and
\begin{equation}
\label{c:eq corollaryJ}
\Big|S^1 \setminus \bigcup_{k=0}^n \tilde{J}_k \Big|= 0\,. 
\end{equation}
Furthermore, denote by $I'_k$ the open subset of $I_k$, such that $\tilde{J}_k =\{r(t)\,:\,t	\in I'_k- \lambda_k\} \subset \tilde{I}_k $; then the system is $\mathcal{F'}$-observable, with  
\begin{equation*}
\mathcal{F'}=\{\{\lambda_k\}\times I'_k\,:\,k\in\mathcal{K}\}\,,
\end{equation*}
and
\begin{equation*}
\sum_{k=0}^n |I'_k| =2\pi\,.
\end{equation*}
\end{corollary}

A couple of helpful remarks may be in order.

\begin{remark}
\label{r:Topt grater than 2pi}
Corollary \ref{corollaryJ} shows that $T_{opt}\geq2\pi$ for any $\mathcal{F}$ as in (\ref{def:F}).
\end{remark}
\begin{remark}
By the proof of Theorem \ref{general theorem} it will be clear that one can generalize to the following situation:
\begin{itemize}
\item The subsets $I_k$, with $k\in \mathcal{K}$, are not a covering of $(0,\infty)$.
\item We can observe on both the endpoints at the same time.
\end{itemize}
To state this result, let each of $\{I^0_{k'}\,:\,k'\in \mathcal{K}'\}$, $\{I^{\pi}_{k''}\,:\,k''\in \mathcal{K}''\}$ be a finite collection of disjoint subintervals of $(0,\infty)$. We may assume that $\sup I^{0}_{h}<\inf I^{0}_m$ for $h<m$ and a similar assumption for the $I^{\pi}_{k''}$'s.\\
Then we can define the observability sets
\begin{equation*}
\mathcal{B}:=\{\{0\}\times I^0_{k'}\,:\,k'\in \mathcal{K}'\}\cup \{\{\pi\}\times I^{\pi}_{k''}\,:\,k''\in \mathcal{K}''\}\,.
\end{equation*}
Furthermore, Definition \ref{def:observability1} readily extends to this situation. Define, 
\begin{align*}
\tilde{I}_{k'}^0&:=\{r(t)\,:\,t\in I_{k'}^{0}\}\,,&\quad& k'\in \mathcal{K}'\,,\\
\tilde{I}_{k''}^{\pi}&:=\{r(t)\,:\,t\in I_{k''}^{\pi}-\pi\}\,,&\quad& k''\in \mathcal{K}''\,.
\end{align*}
Then, the system is (\ref{omo1})-(\ref{omo2}) $\mathcal{B}$-observable if and only if
\begin{equation*}
\Big|S^1\setminus \left(\bigcup_{k'\in \mathcal{K}'\,,\,k''\in \mathcal{K}''}\tilde{I}_{k'}^0\cup \tilde{I}_{k''}^{\pi}\right)\Big|=0\,.
\end{equation*}
We refrain from writing this as a first formulation since we are mainly interested in alternating observations.
\end{remark}

\begin{remark}
Condition (\ref{general condition}) should be compared to the one in Definition 4.1 of \cite{Variable}. It can be seen as the requirements needed to observe the complete energy of any solution.
\end{remark}

Next, we focus on the applications of Theorem \ref{general theorem} and its Corollary \ref{corollaryJ} in particular cases in order to show the potentiality of these results.

\subsubsection{Single exchange of the control position}
\label{ss:single}
In this subsection we analyse the case $\mathcal{K}=\{0,1\}$ and 
\begin{equation}
\label{intervalcase1}
I_0=(0,T_0)\,,\qquad I_1=(T_0,+\infty)\,,
\end{equation}
where $T_0\in(0,2\pi)$.
\begin{corollary}
\label{c:single exchange}
In the previous situation, for all $\pi<T_0<2\pi$ the optimal control time is $T_{opt}=3\pi$ and the system is $\mathcal{F'}$-observable with
\begin{equation} 
I'_0=(0,T_0)\,,\qquad I'_1=(\pi +T_0,3\pi)\,.
\end{equation}
For all $0<T_0\leq \pi$, the optimal control time is $T_{opt}=2\pi +T_0$ and the system is $\mathcal{F'}$-observable with
\begin{equation}
I'_0=(0,T_0)\,,\qquad I'_1=(2T_0,2\pi +T_0)\,,
\end{equation}
or
\begin{equation}
I'_0= \emptyset\,,\qquad I'_1=(T_0,2\pi +T_0)\,.
\end{equation}
\end{corollary}
\subsubsection{Exchange of the control position at a constant time rate}
\label{ss:costantrate}
Then we look at another application of Theorem \ref{general theorem}. We suppose that $\mathcal{K}=\0$ and
\begin{equation*}
I_k=(k T_0,(k+1)T_0)\,,\qquad k\in \0\,,
\end{equation*} 
where $T_0$ is a fixed number in $(0,2\pi)$. So, in this case each interval $I_k$ has the same length. Moreover, in this case
\begin{equation}
\label{constant rate int}
\mathcal{F}= \{\{\lambda_k\}\times (k T_0,(k+1)T_0)\,|\, k\in \0\}\,.
\end{equation}
\begin{lemma}
\label{l:T0 large}
 In the previous situation (i.e. $\mathcal{F}$ is as in (\ref{constant rate int})), suppose that $T_0\in(\pi,2\pi)$. Let $n \in \N$ be the unique integer such that
 \begin{equation}
 \label{Ngrande}
\frac{n+2}{n+1}\pi \leq T_0 <\frac{n+1}{n}\pi\,;
 \end{equation}
then the system is $\mathcal{F}$-observable with
 \begin{equation}
 \label{Tottgrande}
 T_{opt}= (n+2)\pi\,.
 \end{equation}
\end{lemma} 
By the previous Lemma one may infer that if $T_0$ approaches $\pi$ from above then the optimal control time blows up. Indeed, as the next lemma shows, when $T_0=\pi$ the system cannot be observed. 
\begin{lemma}
\label{T0particolare}
Let $\mathcal{F}$ be as in (\ref{constant rate int}). Then:
\begin{enumerate}[label=\roman*)]
\item if $T_0=\pi/2m$, with $m\in \N$, then the system is $\mathcal{F}$-observable with $T_{opt}=2\pi$;
\item if $T_0=\pi/(2n+1)$, with $n\in \0$ then the system is not $\mathcal{F}$-observable in any $T$;
\item if $T_0=2\pi/(2h+1)$, with $h\in 2\N$, then the system is $\mathcal{F}$-observable with $T_{opt}=3\pi - T_0/2$;
\item if $T_0=2\pi/(2h+1)$, with $h\in 2\0+1$, then the system is $\mathcal{F}$-observable with $T_{opt}=3\pi$.
\end{enumerate}
\end{lemma} 
Before proceeding further in the investigation, we introduce some notations
\begin{equation}
\label{def:N,j}
N:=\Big\lfloor\frac{2\pi}{T_0} \Big\rfloor\,,\qquad j:=\Big\lfloor\frac{\pi}{T_0} \Big\rfloor\,;
\end{equation}
here we denote by $\lfloor x \rfloor$ the integer part of $x\in(0,+\infty)$. Instead, $\lceil x\rceil$ denotes the least integer greater than or equal to $x\in(0,+\infty)$.\\
The following lemma, which complements Lemma \ref{T0particolare}, shows that the system is $\mathcal{F}$-observable for all $T_0$ not of the form $\pi/(2n+1)$. It can be easily checked that all the alternatives appearing in the statement are mutually excluding.
\begin{lemma}
\label{T0generale}
Let $\mathcal{F}$ be as in (\ref{constant rate int}), suppose that $T_0\in (0,\pi)$ and $2\pi/T_0 \notin \N$, then the system is $\mathcal{F}$-observable. Furthermore, we have the following:\\
$i)$ if $N,j$ are even, then $T_{opt}=2\pi + jT_0\,;$\\
$ii)$ if $N$ is even and $j$ is odd, then:
\begin{itemize}
\item if $h^*=k^*+1$, then $T_{opt}=\pi + T_0(Nk^*+1)$,
\item if $h^*=k^*$, then $T_{opt}=\pi + T_0(Nk^*-j+1)$,
\end{itemize}
where
\begin{equation*}
h^{*}=\Big\lceil  \frac{ \pi-jT_0+T_0 }{2\pi - NT_0}\Big\rceil\,,
\qquad
k^* = \Big\lceil  \frac{ T_0 }{2\pi - NT_0}\Big\rceil\,;
\end{equation*}
$iii)$ if $N$ is odd and $j$ is even, then
\begin{itemize}
\item if $l^*=q^*$, then $T_{opt}=2\pi q^* + jT_0$,
\item if $q^*+1=l^*$, then $T_{opt}=2\pi q^*+jT_0+\pi$,
\end{itemize}
where
\begin{equation*}
l^*=\Big\lceil \frac{T_0}{(N+1)T_0 - 2\pi}  \Big\rceil\,,\qquad
q^* =\Big\lceil  \frac{\pi - j T_0}{ (N+1)T_0 - 2\pi}  \Big\rceil\,;
\end{equation*}
$iv)$ if $N,j$ are odd, then
$
T_{opt}=NT_0+\pi\,.
$
\end{lemma}

\subsection{Multi-dimensional case ($d>1$)} 
\label{s:multidimensional}
Now we turn to the multidimensional case. Since the Fourier method applied in the one dimensional case does not seem suitable for the multidimensional case (see \cite{Fourier}), a multiplier technique can instead be applied to prove observability (see \cite{multiplier} or \cite{Lions}).\\
The multiplier method permits us only to formulate sufficient conditions in order to obtain observability, so we will not provide an equivalent condition for the observability as we have done in the one dimensional case. Instead, we will show that if the observability set can be written in a special form (see (\ref{Sigma phi}) below) then the system is exactly observable.\\
Let $\Sigma\subset \partial \Omega \times (0,T)$ be an $\mathcal{H}^{d-1}\otimes \mathcal{L}^1$-measurable set (here $\mathcal{H}^{d-1}$, $\mathcal{L}^1$ are respectively the $d-1$ dimensional Hausdorff measure and the $1$ dimensional Lebesgue measure). In analogy with the one-dimensional case we are interested in the following property.
\begin{definition}[$\Sigma$-observability]
\label{def:observabilitym}
Suppose that $\overline{\Sigma} \cap (\partial \Omega \times \{T\})\neq \emptyset$. We say that the system is $\Sigma$-observable in time $T$ if there exists a positive constant $C$ such that
\begin{equation}
\label{dis observabilitym}
\int_{\Sigma} |\partial_{\nu} u |^2\,d(\mathcal{H}^{d-1}\otimes \mathcal{L}^1) \geq C \left( \| u_0\|_{H^1_0(\Omega)}^2 + \| u_1 \|_{L^2(\Omega)}^2\right)\,;
\end{equation}
for all $(u_0,u_1)\in H^1_0(\Omega)\times L^2(\Omega)$ and $u$ is the unique weak solution of (\ref{omo1})-(\ref{omo2}) for the initial data $(u_0,u_1)$. \\
The set $\Sigma$ in (\ref{dis observabilitym}) is called the \textit{observability set}. 
\end{definition}
Of course the validity of the property in Definition \ref{def:observabilitym} depends strongly on the subset $\Sigma \subset \partial \Omega \times (0,T)$. However, in the following, we do not explicitly refer to $\Sigma$ when it is clear from the context.\\ 
We remark that this definition is consistent with Definition \ref{observability0}, but we prefer to give it independently, to stress the topological differences of the observation subset between the one and the multi-dimensional case.\\
Next we introduce some notations. Let $\varphi:[0,T]\rightarrow \mathbb{R}^d$ be a piecewise continuously differentiable and continuous curve of finite length, i.e., 
\begin{equation}
\label{length} 
L(\varphi):=\int_{0}^T |\varphi'(t)|\,dt \,< + \infty\,.
\end{equation}
Then define
\begin{align}
\label{Gamma_t}
\Gamma_{\varphi}(t)&=\{x\in \partial \Omega\,|\,(x-\varphi(t))\cdot \nu >0\}\,,\quad t\in(0,T)\,,\\
\label{Sigma phi}
\Sigma_{\varphi} &= \bigcup_{t\in(0,T)}  \Gamma_{\varphi}(t)\times \{t\}\,,\\
\label{ci}
c_i&=\max_{\overline{\Omega}}|x-\varphi(i)|\,, \qquad i=0,T\,.
\end{align}
Furthermore, let $P=\{t_{-1},\dots,t_N\,:\,0=t_{-1}<t_0<\dots<t_{N-1}<t_N:=T\}$ be a partition of the interval $[0,T]$ and we denote by $\delta(P):=\sup_{j=0\,\dots,N}\{|t_j-t_{j-1}|\}$ the amplitude of the partition $P$. Furthermore, we set
\begin{align}
\label{xphij}
x_{\varphi,j}&:=\varphi(t_{j-1})\,,\\
\label{Gammaj}
\Gamma_{\varphi,j} &:=\{x\in \partial \Omega\,|\,(x-x_{\varphi,j})\cdot \nu >0\}\,,
\end{align} 
for all $j=0,\dots,N$. For future convenience, for a partition $P$ as above, we define 
\begin{equation}
\label{Sigmaphij}
\Sigma_{\varphi}^P:=\bigcup_{j=0}^{N} \Gamma_{\varphi,j} \times (t_{j-1},t_{j})\,.
\end{equation}
In the following we suppose that $\Sigma_{\varphi}$ is $\mathcal{H}^{d-1}\otimes \mathcal{L}^1$-measurable.
\begin{theorem}[Variable Support Observability]
\label{t:variable theorem}
Under the above hypotheses, suppose there exists a sequence of partitions $\{P_k\,:\,k\in\N\}$ of the interval $[0,T]$, such that $\delta(P_k) \searrow 0$ as $k\nearrow \infty$ and
\begin{equation}
\label{convergence}
\lim_{k\rightarrow\infty}\left(\mathcal{H}^{d-1}\otimes \mathcal{L}^1\right)\left(\Sigma_{\varphi} \Delta \Sigma_{\varphi}^{P_k} \right)= 0\,,
\end{equation}
where $A\Delta B :=(A\setminus B) \cup (B\setminus A)$. Furthermore, suppose that $T$ verifies
\begin{equation}
\label{Tvariable}
T > c_0 + \int_{0}^T |\varphi'(t)|\,dt+c_T\,.
\end{equation}
Then the system is $\Sigma_{\varphi}$-observable in time $T$ (see Definition \ref{def:observabilitym}).
\end{theorem}

In Section \ref{s:multidimensional proof} we will provide an extension of the previous Theorem, when $\varphi$ has only a bounded variation on $[0,T]$, see Theorem \ref{t:variable BV}. Indeed, Theorem \ref{t:variable theorem} can be viewed as a corollary of Theorem \ref{t:variable BV} (see Remark \ref{r:BV}), but we prefer to state it independently since it is sufficient in many applications.
\begin{remark}
We do not address the question of the controllability for $T=c_0 + \int_{0}^T |\varphi'(t)|\,dt +c_T$. In general such a $T$ is not expected to be optimal (see Section \ref{s:applications}). Controllability for the threshold optimal time is a tough question even in the case of constant control set (see \cite{Labeau-Rauch} or \cite{Fourier}).
\end{remark}
\paragraph{\textbf{Overview}}
The article is divided into five parts, in Section \ref{s:onedimensional} we prove the results stated in Section \ref{s:alternating}; next we obtain from this the results stated in Subsections \ref{ss:single}-\ref{ss:costantrate}.\\
In Section \ref{s:dependence}, with the results of Subsection \ref{ss:costantrate}, we analyse the dependence of the optimal control time $T_{opt}$ on the length of the interval $T_0$, proving that the map $T_0\mapsto T_{opt}$ has a countable set of discontinuity.\\
In Section \ref{s:multidimensional proof} we prove the results of Subsection \ref{s:multidimensional} and we give a refinement of Theorem \ref{t:variable theorem} when $\varphi$ is only a map of bounded variation on $[0,T]$ (see Theorem \ref{t:variable BV}).\\
In the last Section, we work out some applications of the Theorems stated above and we discuss some problems and further developments.\\
\\
As a conclusion to this section, we want to emphasize that to each observability result follows (by duality) a controllability theorem; indeed the approach of \cite{Fourier} or \cite{Lions} extends to our setting.\\
It is worthwhile 
to give a flavour of this, we state explicitly the controllability result corresponding to Theorem \ref{t:variable theorem} (recall that $\Sigma_{\varphi}$ is assumed to be $\mathcal{H}^{d-1}\otimes \mathcal{L}^1$-measurable).
\begin{theorem}
\label{control}
Let the hypotheses of Theorem \ref{t:variable theorem} be satisfied. Then for all $(v_0,v_1) \in L^2(\Omega)\times H^{-1}(\Omega)$ there exists a control $w\in L^2(\Sigma_{\varphi})$ such that the weak solution $v$ of the following differential problem
\begin{align} 
\label{controllato1}
v_{tt} -  \Delta v&=0\,, &\qquad& (x,t)\in \Omega\times(0,T)\,,\\
v(x,t)&=w(x,t)\,, &\qquad& (x,t)\in \Sigma_{\varphi}\,,\\
v(x,t)&=0\,, &\qquad& (x,t)\in \partial \Omega\times (0,T)\setminus\Sigma_{\varphi}\,,\\
v(x,0)&=v_0(x)\,,&\qquad& x\in \Omega\,,\\
\label{controllato2}
v_t(x,0)&=v_1(x)\,,&\qquad& x\in \Omega\,,
\end{align}
satisfies
\begin{equation*}
v(x,T)=0\,, \qquad v_t(x,T)=0\,.
\end{equation*}
\end{theorem}  
For the weak formulation of the problem (\ref{controllato1})-(\ref{controllato2}) one can use the same argument of \cite[p. 21]{Fourier}; see \cite{ECC} for an explicit formulation.\\ 
The proof of Theorem \ref{control} consists in an application of the HUM method; see \cite{Fourier,Lions} for details. Both formulation and proofs are in fact given in \cite{Fourier,Lions} for observability sets of the form $\Sigma=\Gamma\times (0,T)$, but extend to our case in a straightforward way. The reader can readily deduce analogous results corresponding to Theorems stated above in the one dimensional case.
\section{Preliminary results}
\label{s:prelim}
  We recall the following well posedness result which can be readily verified by semigroup methods (see e.g. \cite{Pazy}).
\begin{proposition}[Well Posedness] 
\label{prop:well posedness}
With the notations introduced in Section \ref{s:introduction}, the following holds:\\
$i)$ Given  
\begin{equation}
\label{data weak}
u_0 \in H^1_0(\Omega)\,,\qquad u_1\in L^2(\Omega)\,,
\end{equation}
the problem (\ref{omo1})-(\ref{omo2}) admits a unique \textit{weak} solution satisfying 
\begin{equation}
\label{sol weak}
u \in C ([0,\infty);H^1_0(\Omega))\cap C^1([0,\infty);L^2(\Omega))\,.
\end{equation}
$ii)$ Given 
\begin{equation}
\label{data strong}
u_0 \in H^2(\Omega)\cap H^1_0(\Omega)\,,\qquad u_1\in H^1_0(\Omega)\,,
\end{equation}
the problem (\ref{omo1})-(\ref{omo2}) admits a unique \textit{strong} solution satisfying
\begin{equation}
\label{sol strong}
u \in C([0,\infty);H^2(\Omega) \cap H^1_0(\Omega))
\cap C^2([0,\infty);L^2(\Omega))\,.
\end{equation}
For any weak solution $u$ as in (\ref{sol weak}), we define the energy $E(t)$ at time $t$ as
\begin{equation}
E(t)= \frac{1}{2} \int_{\Omega} (|\nabla u(x,t)|^2 + |u_t(x,t)|^2)\,dx\,;
\end{equation}
then $E(t)$ is constant and is equal to the \textit{initial energy}
\begin{equation}
\label{energy}
E_0 =\frac{1}{2} \int_{\Omega} (|\nabla u_{0}(x)|^2 + |u_1(x)|^2)\,dx\,.
\end{equation} 
\end{proposition}
We note that if $u$ is a weak solution of problem (\ref{omo1})-(\ref{omo2}) then the trace of the normal derivative $\partial_{\nu}u$ on $\partial \Omega$ is not well defined in the sense of traces of Sobolev functions (see \cite{Evans} or \cite{Grisvard}). For future discussions, the following result is very useful.
\begin{theorem}[Hidden Regularity]
\label{t:hidden}
Let $\Omega\subset \mathbb{R}^d$ be a domain of class $C^2$ or convex. Then for any $T>0$ and $(u_0,u_1)\in H^1_0(\Omega)\times L^2(\Omega)$, the map
\begin{align*}
H^1_0(\Omega)\times L^2(\Omega) &\rightarrow L^2(\partial \Omega\times (0,T))\,,\\
(u_0,u_1)&\mapsto \partial_{\nu} u\,,
\end{align*}
is well defined, continuous and coincides with the classical trace (in the sense of Sobolev spaces) if $(u_0,u_1)\in (H^2(\Omega)\cap H^1_0(\Omega)) \times H^1_0(\Omega)$ since $u$ is as in (\ref{sol strong}).
\end{theorem}
In case $\Omega$ is a $C^2$ set, the proof can be found in \cite{multiplier}, \cite{Hidden} or \cite{Lions}; in the one dimensional case Theorem \ref{t:hidden} can be proven also by the use of Fourier series (see \cite[Chapter 1]{Fourier}).\\
To our knowledge, no proof can be found in the case of convex  domains $\Omega$; so we will sketch the proof.
In the $C^2$-case, the proof in \cite{multiplier} starts with using explicitely the smoothness assumption, and it does not extend to convex domains. Our approach is based on an identity proved by the authors in \cite{ECC} and some well known regularity results for the Laplace operator on convex domains.
\begin{proof}[Proof of Theorem \ref{t:hidden} when $\Omega$ is convex]
Let us start with 
$$u_0,u_1\in\text{Span} \{e_k\,|\,k\in \0\}$$ 
where $e_k$ are the eigenfuctions of the Laplacian operator with Dirichlet boundary condition on $\Omega$. Regularity theory for elliptic PDEs ensures that $e_k\in H^2(\Omega)\cap H^1_0(\Omega)$; see \cite{Grisvard}.\\
Thus, the strong solution $u$ of (\ref{omo1})-(\ref{omo2}) is as in (\ref{sol strong}) and $\partial_{\nu}u$ is well defined in the sense of the trace for Sobolev spaces; see \cite{Evans}. Furthermore, since $u\in C([0,T];H^2(\Omega)\cap H^1_0(\Omega))$ then $\partial_{\nu}u \in C([0,T];L^2(\partial\Omega))$.\\
To prove the assertion we use the following equality (see \cite[Lemma 1.8]{ECC}),
\begin{multline}
\label{hidden1}
\frac{1}{2}\int_{\partial \Omega\times(0,T)} (x-x_0)\cdot \nu \,|\partial_{\nu} u(x,t)|^2 d \mathcal{H}^{d-1}(x)\otimes dt =\\
\left[ \int_{\Omega} u_t(x,t) \left( \nabla u(x,t) \cdot (x-x_0) + \frac{d-1}{2} u(x,t) \right)\,dx\right]_{0}^{T} +
 T\,E_0\,;
\end{multline}
where $x_0$ is an arbitrary point of $\mathbb{R}^d$.\\
%
Note that, since each convex function is locally lipschitz (see \cite[Theorem 6.7]{Evans}), then the exterior normal derivative $\nu$ is well defined $\mathcal{H}^{d-1}$-a.e. on $\partial \Omega$. From the convexity, it follows that for each $x_0 \in \Omega$ then $(x-x_0)\cdot\nu \geq \delta >0$ $\mathcal{H}^{d-1}$-a.e., for some $\delta$ positive. By (\ref{hidden1}) we have
\begin{multline}
\label{hidden2}
\frac{\delta}{2}\int_{0}^{T}\int_{\partial \Omega} \,|\partial_{\nu} u(x,t)|^2 d \mathcal{H}^{d-1}(x)\,dt \leq\\
\left[ \int_{\Omega} u_t(x,t) \left( \nabla u(x,t) \cdot (x-x_0) + \frac{d-1}{2} u(x,t) \right)\,dx\right]_{0}^{T} +
 T\,E_0\,.
\end{multline}
Using Poincar\'e's inequality, it is clear that the RHS of (\ref{hidden2}) can be dominated by $C\,E_0$, for some positive constant $C$; a standard density argument completes the proof.
\end{proof}

\section{One dimensional case}
\label{s:onedimensional}
In the one dimensional case, as stated in Subsection \ref{s:alternating}, w.l.o.g. we set $\Omega=(0,\pi)$. It is well known that the family $\{\sin (n\,x)\,|\,n\in \N\}$ forms a complete orthogonal system of $L^2(0,\pi)$. Indeed, this family consists of eigenfunctions of the operator $d^2/dx^2$ with Dirichlet condition. Thus, the unique solution $u$ for the initial data $(u_0,u_1)\in H^1_0(0,\pi)\times L^2(0,\pi)$ provided by Proposition \ref{prop:well posedness} is given by
\begin{equation}
\label{monodimensionale Fourier}
u(x,t)=\sum_{j=1}^{\infty}\left[ \hat{u}_0(j) \cos (jt) + \frac{1}{j} \hat{u}_1(j)  \sin(jt) \right] \sin(jx)\,;
\end{equation}
here
\begin{equation}
\hat{u}_{h}(j) = \frac{2}{\pi} \int_0^{\pi} u_{h}(x)\,\sin(j \,x)\,dx\,, \qquad h=0,1\,;\, j\in \N\,.
\end{equation}
By a standard approximation argument, we see that the set
\begin{equation*}
\mathcal{H} = \{(u_0,u_1)\,|\, u_0,u_1 \in  \text{Span}\left(\sin(jx)\,|\,j\in \N\right)\,\}\,,
\end{equation*}
is dense in $ H^1_0(0,\pi)\times L^2(0,\pi)$.\\
The proof of Theorem \ref{general theorem} relies on the following observation: if $(u_0,u_1)\in \mathcal{H}$ then 
\begin{equation}
\label{derivative x}
u_x(x,t) = \sum_{j=1}^N \left[ \hat{u}_0(j) \cos (jt) + \frac{1}{j} \hat{u}_1(j)  \sin(jt) \right] j \cos(jx)\,, 
\end{equation}
and
\begin{multline}
\label{u_x(pi)}
u_x(\pi,t) = \sum_{j=1}^N j \left[ \hat{u}_0(j) \cos (jt) \cos(j \pi) + \frac{1}{j} \hat{u}_1(j)  \sin(jt)\cos(j\pi) \right]\\
= \sum_{j=1}^N j \left[ \hat{u}_0(j) \cos (j(t-\pi)) + \frac{1}{j} \hat{u}_1(j)  \sin(j(t-\pi)) \right]
=u_x(0,t-\pi)\,;
\end{multline}
where $N\in \N$ is sufficiently large and the last inequality in \eqref{u_x(pi)} follows by (\ref{derivative x}) with $x=0$.\\
Now we recall the following result, which is an easy application of the Parseval identity in Hilbert spaces (in this case $L^2(0,\pi)$ and $L^2(0,2\pi)$), for a complete proof we refer to \cite{Fourier}.
\begin{proposition}
\label{prop:observability single endpoint}
Let $u$ be the unique weak solution of (\ref{omo1})-(\ref{omo2}) with $\Omega=(0,\pi)$, then
\begin{equation}
\int_{0}^{2\pi}|u_x(0,t)|^2\,dt =2 (\|  u_0'\|_{L^2(0,\pi)}^2+ \|  u_1\|_{L^2(0,\pi)}^2)\,;
\end{equation}
here $(\cdot)'=d(\cdot)/dx$.
\end{proposition}
Now we are ready to prove Theorem \ref{general theorem}:
\begin{proof}[Proof of Theorem \ref{general theorem}]
We first prove that if the condition (\ref{general condition}) holds then the system is $\mathcal{F}$-observable (see Definition \ref{def:observability1}) where $\mathcal{F}$ is as in (\ref{def:F}); in turn we have to prove inequality (\ref{observability0}).\\
Due to Theorem \ref{t:hidden}, we can suppose that $(u_0,u_1)\in \mathcal{H}$. Now, the left hand side of (\ref{observability0}) is equal to 
\begin{equation}
\label{eq:onedimidentity}
\begin{aligned}
\sum_{k=0}^n \int_{I_k} |u_x(\lambda_k,t)|^2\, dt &= \sum_{k=0}^n \int_{I_k} |u_x(0,t-\lambda_k )|^2\,dt\\
&= \sum_{k=0}^n \int_{I_k - \lambda_k } |u_x(0,t)|^2\,dt \\
&= \sum_{k=0}^n \int_{\tilde{I}_k}  |u_x(0,t)|^2\,dt \asymp \int_{\cup_{k=0}^n \tilde{I}_k} |u_x(0,t)|^2\,dt\,.
\end{aligned}
\end{equation}
Here the first equality follows by (\ref{u_x(pi)}), the second from a change of variable in the integrals.
Then it follows by (\ref{general condition}) that
\begin{equation*}
\sum_{k=0}^n \int_{I_k} |u_x(\lambda_k,t)|^2\, dt \asymp \int_{\cup_{k=0}^n \tilde{I}_k} |u_x(0,t)|^2\,dt = \int_{S^1} |u_x(0,t)|^2\,dt = 4 E_0\,,
\end{equation*}
where the last equality follows by Proposition \ref{prop:observability single endpoint}.\\
Conversely, we show that if (\ref{general condition}) does not hold then the observability inequality fails for some initial data. To do this we need the following:
\begin{lemma}
\label{l:lemma proof general condition}
For each open set $U\subset S^1$, there exists a pair $(u_0,u_1)\in H^1_0(0,\pi)\times L^2(0,\pi)$ different from the null pair, such that the corresponding solution $u$ to the problem (\ref{omo1})-(\ref{omo2}) satisfies
\begin{equation}
\text{supp} \,\, u_x(0,\cdot) \subset U\,.
\end{equation}  
\end{lemma}
We postpone the proof of the Lemma \ref{l:lemma proof general condition} and continue with the previous discussion. Since, by assumption, the condition (\ref{general condition}) fails, then there exists an open subset $U$ contained in $S^1 \setminus \cup_{j=0}^n \tilde{I}_j$. Then choose $(u_0,u_1)\in H^1_0(0,\pi)\times L^2(0,\pi)$ in accordance with Lemma \ref{l:lemma proof general condition}. By \eqref{eq:onedimidentity}, if the observability inequality (\ref{observability0}) holds then we have 
\begin{equation*}
0=\int_{\cup_{k=0}^n \tilde{I}_k} |u_x(0,t)|^2\,dt\asymp \sum_{k=0}^n\int_{I_k} |u_x(\lambda_k,t)|^2\, dt \geq C E_0 \neq 0\,,
\end{equation*}
which is evidently an inconsistency. The last assertion of Theorem \ref{general theorem} follows by the previous argument.
\end{proof}
\begin{proof}[Proof of Lemma \ref{l:lemma proof general condition}]
We can of course suppose that $\pi \notin U$ (otherwise we can replace $U$ with $U\setminus \overline{V}$, where $V$ is a sufficiently small neighborhood of $\pi$). Let $\psi\not\equiv 0$ be a $ C^{\infty}$-function on $(0,2\pi)$ such that $\text{supp}\psi \subset U$ and $\int_0^{2\pi}\psi(t)\,dt=0$; it is enough to prove that there exists a pair $(\tilde{u}_0,\tilde{u}_1)\in H^1_0(0,\pi)\times L^2(0,\pi)$ such that the weak solution of (\ref{omo1})-(\ref{omo2}) $\tilde{u}$ satisfies $\tilde{u}_x(0,\cdot)=\psi$.\\
Firstly, by the Fourier series representation formula (\ref{monodimensionale Fourier}), we obtain
\begin{equation*}
  \tilde{u}_x(0,t) = \tilde{u}'_0(t) + \tilde{u}_1(t)\,,\qquad t\in(0,\pi)\,.
\end{equation*}
Instead for $t\in(\pi,2\pi)$, set $t=2\pi-\tau$ with $\tau \in (0,\pi)$ and using a simple reflection argument we obtain
\begin{equation*}
  \tilde{u}_x(0,2\pi - \tau) =\tilde{u}_0'(\tau) -\tilde{u}_1(\tau)\,, \qquad \tau \in (0,\pi)\,. 
\end{equation*}
Then $\tilde{u}_x(0,t)=\psi(t)$ is equivalent to
\begin{equation*}
  \psi(t)= \tilde{u}_0'(t)+\tilde{u}_1(t)\,, \qquad \psi(2\pi-t) =  \tilde{u}_0'(t) -\tilde{u}_1(t)\,, \qquad t\in(0,\pi)\,.
\end{equation*}
We immediately obtain
\begin{equation*}
  \tilde{u}'_0 (t) = \frac{1}{2} \left( \psi(t ) + \psi(2\pi-t) \right)\,,\qquad t\in (0,\pi)\,,
\end{equation*}
and
\begin{equation}
\label{phi1}
\tilde{u}_1(t) = \frac{1}{2} \left( \psi(t ) - \psi(2\pi-t) \right)\,, \qquad t\in (0,\pi)\,.
\end{equation}
Then 
\begin{equation}
\label{phi0}
\tilde{u}_0 (x) =  \frac{1}{2} \int_{0}^{x}\left( \psi(t ) + \psi(2\pi-t) \right)\,dt\,, \qquad x\in(0,\pi)\,.
\end{equation}
Since by assumption $\psi$ is a smooth function, $\int_0^{2\pi} \psi(t)dt=0$ and $\pi \notin U \supset \text{supp}\,\psi$; then $\tilde{u}_0\in H^1_0(0,\pi)$.\\ 
Thus, the functions $\tilde{u}_0\in H^1_0(0,\pi)$ and $\tilde{u}_1\in L^2(0,\pi)$, defined respectively in (\ref{phi0}) and (\ref{phi1}), provide the desired pair of initial data.
\end{proof}
Next we prove Corollary \ref{corollaryJ}.
\begin{proof}[Proof of Corollary \ref{corollaryJ}] 
It is sufficient to define recursively $\tilde{J}_0:=\tilde{I}_0$ and
\begin{equation*}
\tilde{J}_k := \tilde{I}_k \setminus \bigcup_{j=0}^{k-1} \tilde{J}_j\,,\qquad k=1,\dots,n\,.
\end{equation*}
The other statement in the Corollary follows easily by this.
\end{proof}
Corollary \ref{c:single exchange} is an easy consequence of Theorem \ref{general theorem} or Corollary \ref{corollaryJ}. Before proving the results of Subsection \ref{ss:costantrate} we introduce some conventions.\\
With an abuse of notation, we denote the subset of $S^1$ mapped from the interval $(\alpha,\beta)\subset \mathbb{R}$ through the application $r$ still as $(\alpha,\beta)$. So we will write
$$
(\alpha,\beta) = (\gamma,\delta)\subset S^1\,;
$$ 
if $r(\alpha)=r(\gamma)$, $r(\beta)=r(\delta)$. With this convention if $\alpha\in(0,2\pi)$, $\beta\in(2\pi,4\pi)$, we have
$$
(\alpha,\beta)= (\alpha,2\pi) \bigcup (0,\beta - 2\pi)\,. 
$$
Moreover, in the following we identify the measure space $(S^1,|\cdot|)$ with the space $([0,2\pi),\mathcal{L}^1)$.\\
\\
We begin with the proof of Lemma \ref{T0particolare} which is the simplest one and permits us to explain in details some technique which will be used also to prove Lemma \ref{l:T0 large} and \ref{T0generale}.
\begin{proof}[Proof of Lemma \ref{T0particolare}]
$i)$ 
First of all, we note that if $k$ is even then the observability subset $\tilde{I}_k$ comes from a subset of observation on the endpoint $x=0$, so we obtain
\begin{align*}
\tilde{I}_0 &= (0,T_0),\,\tilde{I}_2 = (2T_0,3T_0),\dots,\tilde{I}_{2m-2} = (\pi-2T_0,\pi - T_0)\,,
\\
\tilde{I}_{2m} &= (\pi,\pi + T_0)\,,\dots,\tilde{I}_{4m-2} = (2\pi - 2T_0,2\pi- T_0)\,;
\end{align*}
instead, for odd indexes, we have
\begin{align*}
\tilde{I}_1 &= (\pi + T_0,\pi + 2 T_0),\dots,\tilde{I}_{2m-1} = (2\pi-T_0,2\pi),\,\\
\tilde{I}_{2m+1}&=(T_0,2T_0)\,,\tilde{I}_{2m+3}=(3T_0,4T_0),\dots,\\
\tilde{I}_{4m-1} &= ((4m-1)T_0 + \pi, 4mT_0 + \pi) = (\pi - T_0,\pi)\,.
\end{align*}
It is easy to see that $(0,2\pi)$ is covered by the following subsets, in the following order:
\begin{equation*}
  \tilde{I}_{0}\,,\,\tilde{I}_{2m+1}\,,\,\tilde{I}_{2}\,,\,\tilde{I}_{2m+3}\,,\,....\,,\,\tilde{I}_{4m-1}\,,\,\tilde{I}_{2m}\,,
\,
  \tilde{I}_{1}\,,\,\tilde{I}_{2m+2}\,,\,\tilde{I}_{3}\,,\,....\,,\,\tilde{I}_{2m-1}\,.
\end{equation*}
%
The condition \eqref{general condition} in Theorem \ref{general theorem} is satisfied, then the system is observable in $T=2\pi$.\\
Finally we note that $T=2\pi$ is optimal, in fact $(\pi-T_0,\pi)$ is covered only by $\tilde{I}_{4m-1}=((4m-1)T_0, 4mT_0) = (2\pi - T_0,2\pi)$ so the optimality follows by Theorem \ref{general theorem}.\\
\\
$ii)$ We proceed as in $i)$, by listing the subsets of observation coming, through the map $r$, from the time interval $(0,2\pi)$. For even indexes we have
\begin{align*}
\tilde{I}_0&= (0,T_0),\,\tilde{I}_2 = (2T_0,3T_0),\dots,\tilde{I}_{2n} = (\pi-T_0,\pi)\,,\dots\,,\\
\tilde{I}_{2n+2}&=(\pi + T_0,\pi + 2T_0)\,, \dots\,,\\
\tilde{I}_{4n} &= ((4n)T_0,(4n+1)T_0)= (2\pi - 2T_0, 2\pi - T_0)\,;
\end{align*}
instead for odd indexes we have
\begin{align*}
\tilde{I}_1 &= (\pi + T_0,\pi + 2 T_0)\,,\dots,\tilde{I}_{2n-1} = (2\pi-2T_0,2\pi-T_0)\,,\dots\,,\\
\tilde{I}_{2n+1}&=(0,T_0)\,,\dots\,,\\
\tilde{I}_{4n+1} &= ((4n+1)T_0 + \pi, (4n+2) T_0 + \pi) = (\pi - T_0,\pi)\,. 
\end{align*}
It is clear that
\begin{equation*}
\tilde{I}_{0}= \tilde{I}_{2n+1},\dots,\tilde{I}_{2n}= \tilde{I}_{4n+1}\,,\,
\tilde{I}_{2n+2}= \tilde{I}_{1},\dots,\tilde{I}_{2n-1}= \tilde{I}_{4n}\,.
\end{equation*} 
Thus, the following subsets
\begin{equation*}
  (T_0,2T_0)\,,\dots\,,((4n+1)T_0,(4n+2)T_0)
\end{equation*}
cannot be covered by other subsets since the subsets $\{\tilde{I}_k\,:\,k\in \0\}$ are $2\pi$-periodic in time. In view of the necessary condition \ref{general condition} in Theorem \ref{general theorem} the system cannot be observable in any time $T$.\\
\\
$iii)$ Reasoning as in $i)-ii)$, we have for even indexes
\begin{align*}
&\tilde{I}_{0}=(0,T_0)\,,\dots\,,\tilde{I}_{h}=(hT_0,(h+1)T_0)= \left(\pi-\frac{T_0}{2},\pi+\frac{T_0}{2}\right)\,,\dots\,,\\
&\tilde{I}_{2h}=(2hT_0,(2h+1)T_0)= (2\pi-T_0,2\pi)\,,
\end{align*}
where we use that $hT_0=\pi - T_0/2$. Instead for odd indexes, we have
\begin{align*}
&\tilde{I}_{1}=(\pi + T_0,\pi + 2 T_0)\,,\dots\,,\,\\
&\tilde{I}_{h-1}=(\pi +(h-1)T_0,\pi +hT_0)= \left(2\pi - \frac{3}{2}T_0,2\pi - \frac{1}{2}T_0\right)\,,\dots\,,\\
&\tilde{I}_{h+1}=\left(\frac{T_0}{2},\frac{3T_0}{2}\right)\,,\dots\,,\\
&\tilde{I}_{2h-1}=((2h-1)T_0-\pi,2hT_0-\pi)= (\pi-2T_0,\pi-T_0)\,.
\end{align*}
Note that the subset 
\begin{equation}
\label{dim 1.3}
\left(\pi-T_0,\pi-\frac{T_0}{2}\right)
\end{equation}
is one of the subsets not covered by any $\tilde{I}_k$ for $k=1$, \dots, $2h$. In particular, we will see that the subset in (\ref{dim 1.3}) would be the last to be covered among the intervals not yet covered.\\
We now proceed further listing the later subsets (in the index or chronological order). The observability subsets coming from the endpoint $x=0$ are
\begin{align*}
\tilde{I}_{2h+2}=(2\pi+T_0,2\pi + 2T_0)&=(T_0,2T_0)\,,\dots\,,\\
\tilde{I}_{3h}= ((3h)T_0,(3h+1)T_0)&=(2\pi + (h-1)T_0,2\pi + h\,T_0)\\
&= \left(3\pi - \frac{3T_0}{2}, 3\pi - \frac{T_0}{2}\right)\\
&=\left(\pi - \frac{3T_0}{2}, \pi - \frac{T_0}{2}\right)\,;
\end{align*}
instead, from the endpoint $x=\pi$ they are
\begin{equation*}
  \tilde{I}_{2h+1}=(\pi,\pi + T_0)\,,\dots\,,\tilde{I}_{3h-1}=\left(2\pi-\frac{5}{2}T_0,2\pi-\frac{3}{2}T_0\right)\,.
\end{equation*}
It is easy to see that $(0,2\pi)$ is covered by the following subsets, which are ordered e.g., according to their left endpoint as indicated below:
\begin{multline*}
  \tilde{I}_0\,,\,\tilde{I}_{2h+2}\,,\,\tilde{I}_2\,,\,\tilde{I}_{2h+4}\,,\,\dots\,,\,\tilde{I}_{3h}\,,
  \,\tilde{I}_h\,,\,\tilde{I}_{2h+1}\,,\\
  \tilde{I}_{1}\,,\,\tilde{I}_{2h+3}\,,\,\tilde{I}_3\,,\,\dots\,,\,\tilde{I}_{3h-1}\,,\,\tilde{I}_{h-1}\,,\,\tilde{I}_{2h}\,.
\end{multline*}
To compute the optimal control time, we note that the subset with the highest index in the previous sequence is $\tilde{I}_{3h} = r(I_{3h})$; then, by Theorem \ref{general theorem}, $T_{opt}\leq (3h+1)T_0=3\pi - T_0/2$. Moreover, the set 
$$\tilde{I}_{3h}\setminus r((3hT_0,t))\subset(\pi-T_0,\pi-T_0/2)\,,$$ 
is not covered for any $t<(3h+1)T_0=3\pi - T_0/2$. Thus, by Theorem \ref{general theorem}, we have $T_{opt}\geq 3\pi - T_0/2$ and this concludes the proof.\\
\\
$iv)$ The proof is similar to $iii)$. 
In this case, since $h$ is odd, the observability sets $\tilde{I}_{k}$, for $k\leq 2h$, do not cover the subset
\begin{equation*}
  (hT_0,\pi)\,,
\end{equation*}
which is different from the set in (\ref{dim 1.3}) but plays the same role of that set in $iii)$.\\
Note that $\tilde{I}_{3h+1}\supset (hT_0,\pi)$. By Theorem \ref{general theorem}, this implies 
$$T_{opt}\leq (3h+2)T_0=2\pi + \frac{h+1}{2h+1}2\pi\,.$$
Thus, we can use only a portion of $\tilde{I}_{3h+1}$ to cover $(hT_0,\pi)$ (the red part). Indeed,
\begin{equation*}
  \tilde{I}_{3h+1}\supset r((3h+1)T_0,3\pi)) = \left( hT_0,\pi \right)\,.
\end{equation*}
Now the conclusion follows as in $iii)$.
\end{proof}

The previous lemma shows that the $\mathcal{F}$-observability is not expected in general. The main idea is to ``follow" the covering of $S^1$ with $\tilde{I}_k$ when $k$ grows. This will be exploit in the proof of Lemma \ref{l:T0 large} below.

\begin{proof}[Proof of Lemma \ref{l:T0 large}]
Since $\tilde{I}_0=(0,T_0)$, for the $\mathcal{F}$-observability it is sufficient to prove that the subsequent observation sets cover $(T_0,2\pi)$. 

The subset $\tilde{I}_1$ is shifted towards the endpoint $t=2\pi$ with a gap equal to $T_0-\pi>0$; more generally each $\tilde{I}_k$ has the same behaviour with a gap equal to $k(T_0-\pi)$.\\ 
Thus, we are able to cover the missing interval $(T_0,2\pi)$ in a finite time. Following our notation, we can check that for each $k\in \N$,
\begin{equation*}
\tilde{I}_k=(kT_0-k\pi,(k+1)T_0-k\pi)\,.
\end{equation*}

For the optimality, we are interested in the first interval $\tilde{I}_k$ which ``touches" the endpoint $t=2\pi$. More precisely, we denote with $n$ the least integer $k\geq 1$ such that,
\begin{equation*}
(k+1)T_0-k\pi \geq 2\pi \quad \Leftrightarrow \quad  T_0\geq \frac{(k+2)\pi}{(k+1)}\,.
\end{equation*}
By minimality, we have $nT_0-(n-1)\pi<2\pi$, i.e. $T_0<\pi((n+1)/n)$. In turn, $n$ is the unique integer such that
\begin{equation*}
\frac{n+2}{n+1}\pi \leq T_0 <\frac{n+1}{n}\pi\,;
\end{equation*}
which is the condition (\ref{Ngrande}). \\
We now turn to prove the formula for $T_{opt}$. Note that, we do not use all the subset $\tilde{I}_n$ to cover the missing part $(T_0,2\pi)$, since the blue part is an extra observation set not useful to cover $(T_0,2\pi)$. So, the optimal control time is the least $t\geq nT_0$ such that
\begin{equation*}
r((nT_0,t))=(nT_0-n\pi,t-n\pi)\supset (T_0,2\pi)\,.
\end{equation*}
This is equivalent to $t-n\pi \geq 2\pi$, i.e. $t\geq (n+2)\pi$; so $T_{opt}=(n+2)\pi$.
\end{proof} 

The proofs of Lemmas \ref{l:T0 large}-\ref{T0particolare} contain all the ingredients needed to prove Lemma \ref{T0generale}. Indeed, in $i)$ and $iv)$ we exploit the same strategy of Lemma \ref{T0particolare}, i.e. we list the subsets of observation $\tilde{I}_k$ when $k$ grows and we follow the covering of the ``missing subsets" of $S^1$. In such cases, the strategy works since we will discover that $T_{opt}$ is``relatively small". More precisely, we will see $T_{opt}\leq 3\pi$. 

Such strategy is no more efficient in cases $ii)$-$iii)$. On the contrary, the proof is based on a covering analysis of the ``missing sets" of $S^1$, as done in Lemma \ref{l:T0 large} for the set $(T_0,2\pi)$. In Section \ref{s:dependence}, we will see that $T_{opt}$ can be arbitrarily large provided $T_0$, such that $2\pi/T_0\not\in \N$, is suitable. Of course, such $T_0$ verifies $ii)$ or $iii)$.  

\begin{proof}[Proof of Lemma \ref{T0generale}] $i)$ By assumption, we have $N=2j$ and $j\geq 2$. By Remark \ref{r:Topt grater than 2pi} we have $T_{opt}\geq 2\pi$. The observability sets for $t\leq 2\pi$ are 
\begin{equation*}
\tilde{I}_{i}\,,\qquad i=0\,,\dots\,,N\,.
\end{equation*}
With simple considerations, one can check that the above subsets do not cover $S^1$. Namely the following subsets are not covered
\begin{align}
\label{eq:J1}
J_k &:=\left( (j+2)T_0-\pi + kT_0,(2+k)T_0 \right)\,,\\
\label{eq:J2}
J_{j+k}&:=\left((j+1)T_0+kT_0,\pi+(k+1)T_0\right)\,;
\end{align}
where $k=0,2,\dots,j-2$.\\
So, we have to look at the subsequent observability set (in the index order) to show the $\mathcal{F}$-observability. 
Observe that, the first observability set $\tilde{I}_{k}$ with $k\geq 2j$ which intersects $J_0$ is $\tilde{I}_{N+2}$. Moreover,
\begin{multline*}
 \tilde{I}_{N+ 2} \cap J_0 = (\max\{NT_0 + 2T_0 -2\pi,(j+2)T_0-\pi \},2T_0)\\
 =((j+2)T_0-\pi,2T_0) =J_0\,.
\end{multline*}
In fact the second last equality follows by
\begin{equation*}
NT_0 + 2T_0 - 2\pi < (j+2)T_0-\pi \Leftrightarrow (N-j)T_0=jT_0 < \pi\,;
\end{equation*}
and the last inequality is verified (recall the definition (\ref{def:N,j})).\\
Owing to the periodical character of our setting, the same reasoning shows that
\begin{equation*}
\tilde{I}_{N+2+m}\supset J_{m}\,;\qquad m=0,2,\dots,j-2\,.
\end{equation*}
Furthermore, the least time which allows us to cover the subsets $J_{m}$, for $m=0,2,\dots,j-2$, is the least one which ensures to cover $J_{j-2}$ and it is the minimum $t'$ such that
\begin{align*}
r((3jT_0,t'))= (3jT_0-2\pi,t'-2\pi)\supset J_{j-2}=(2jT_0-\pi,jT_0)\,.
\end{align*}
The previous condition is equivalent to $t'-2\pi \geq jT_0$, then $2\pi+jT_0$ is the least time which allows us to cover $J_m$ for $m=0,2,\dots,j-2$.\\
Now, we show that we can cover the subsets $J_m$ with $m=j,j+2,\dots,2j-2$.\\
Again by periodicity, we have only to prove that we can cover $J_{2j-2}$, the previous $J_m$'s are then automatically covered by a simple translation argument. Moreover, it is easy to see that $\tilde{I}_{3j-1}\supset J_{2j-2}$ and the above argument shows that the least time which allows us to cover $J_{2j-2}$ is $2\pi+(j-1)T_0$.\\
Finally, $T_{opt}$ is the least time which allows us to cover \textit{any} missing sets $J_m$'s. By the previous considerations, we have $T_{opt}=2\pi+jT_0$.\\ 
\\
$ii)$ As in $i)$, we have $N=2j$. Reasoning as in $i)$, the following subsets are not covered by $\tilde{I}_{1},\dots,\tilde{I}_N$:
\begin{align}
\label{eq:L1}
L_h&:=((h+1)T_0,(j+2+h)T_0-\pi)\,,\\
\label{eq:Lintermediate}
L_{j-1}&:=(jT_0,(j+1)T_0)\,,\\
\label{eq:L2}
L_{j+1+h}&:=((2+h)T_0+\pi,(j+h+3)T_0)\,;
\end{align}
where $h=0,2,\dots,j-3$. Here, we have assumed $j\geq 3$, the case $j=1$ is simpler and follows by the same argument. Note that, in case $j=1$ the only subset of $S^1$ not covered is $L_{0}:=(T_0,2T_0)$.\\
The idea is to use an argument similar to the one used in Lemma \ref{l:T0 large} on each $L_k$.
Roughly speaking, the clue is that for each $\tilde{I}_k$, ``the corresponding one" in the next period (i.e. after a time $2\pi$) is $\tilde{I}_{k+N}$ and it is ``shifted" by a quantity $2\pi-NT_0$ towards the endpoint $t=0$ with respect to $\tilde{I}_k$ (recall that we have identified the measure space $(S^1,|\cdot|)$ with $([0,2\pi),\mathcal{L}^1)$). A similar argument holds for the subsequent time periods; indeed $\tilde{I}_{k+nN}$ is ``shifted" by a quantity $n(2\pi-NT_0)$ towards the endpoint $t=0$ with respect to $\tilde{I}_k$. This fact will be used to prove that the missing sets in (\ref{eq:L1})-(\ref{eq:L2}) can be covered in a finite time.\\
Furthermore, since $N$ is even, if $\tilde{I}_{k}$ comes from an observation on the endpoint $x=0$ (resp. $x=\pi$) then $\tilde{I}_{k+nN}$ still comes from an observation on $x=0$ (resp. $x=\pi$); this will be useful in the following.\\
Again by periodicity as in the proof of point $i)$, by a shift argument, one can reduce the proof to show that the subsets $L_{j-1}$ and $L_{2j-2}$ are covered in finite time. Furthermore, the least time $t$ which allows us to cover both $L_{j-1}$ and $L_{2j-2}$ will be the optimal control time.\\
For the reader's convenience, we divide the proof into two steps. In the first one, we determine the least $\k\in \N$ which allows us to cover both $L_{j-1},L_{2j-2}$ with subsets in $\{\tilde{I}_k\,:\,k\leq \k\}$. In the second one, we compute $T_{opt}$ by means of Theorem \ref{general theorem}.

\textbf{Step 1}. We first analyse the covering of $L_{j-1}$. By the previous discussion, to cover $L_{j-1}$ we can consider two families of subsets, namely $\H_1:=\{\tilde{I}_{mN+j+1}\,:\,m\in \N\}$ and $\H_2:=\{\tilde{I}_{1+mN}\,:\,m\in \N\}$. Note that the intervals in $\H_1$ come from $x=0$, while ones in $\H_2$ come from $x=\pi$. Let us compute the least $k_i\in \N$ (with $i=1,2$) such that  the subsets in $\{\tilde{I}_k\in \H_i\,:\,k\leq k_i\}$ cover $L_{j-1}$.\\

First consider the family $\H_1$. As usual, we ask for the minimum $m^*>1$ such that the following inclusion holds
\begin{multline*}
\bigcup_{m=0}^{m^*}(\tilde{I}_{mN+j+1}\cap L_{j-1})=\\
(\max\{(m^*N+j+1)T_0-2m^*\pi,jT_0\},(j+1)T_0)\supset L_{j-1}\,.
\end{multline*}
By \eqref{eq:Lintermediate}, this occurs if and only if $(m^*N+j+1)T_0-2m^*\pi \leq jT_0$. Thus $m^*=k^*:=\lceil T_0/(2\pi-NT_0)\rceil$ and
$$
k_1=k^*N+j+1\,.
$$
A similar argument is valid for the family $\H_2$. One discovers that  $k_2=1+h^*N$, where
\begin{equation}
\label{eq:hstarNparijdispari}
h^*:=\Big\lceil\frac{\pi-jT_0+T_0}{2\pi-NT_0}\Big\rceil=\Big\lceil \frac12 + \frac{T_0}{2\pi-NT_0}\Big\rceil\,.
\end{equation}
All together, $\tilde{k}:=\min\{k_1,k_2\}$ is the least integer $h$ such that $\{\tilde{I}_k\,:\,k\leq h\}$ covers $L_{j-1}$.
For $L_{2j-2}$ a similar argument applies with respect to the families $\{\tilde{I}_{N+mN}\,:\,m\in \N\}$ and $\{\tilde{I}_{j+mN}\,:\,m\in \N\}$. One readily checks that, for any $T_0$, the family $\{\tilde{I}_k\,:\,k< \tilde{k}\}$ covers $L_{2j-2}$. This implies $\k=\tilde{k}$.
\textbf{Step 2}. We are now in a position to compute $T_{opt}$. Below $k_i$ (with $i=1,2$) are as in \textbf{Step 1}. By \eqref{eq:hstarNparijdispari} one has $k^*\leq h^*\leq k^*+1$, thus we can divide the analysis into two cases:

\begin{itemize}
\item If $h^*=k^*+1$, then $k_1< k_2$ and $\k=k_1$. One can check that,
\begin{equation*}
T_{opt}=2k^* \pi + (N+1)T_0-\pi + (k^*-1)(NT_0-2\pi)=\pi+T_0(Nk^*+1)\,.
\end{equation*}

\item  If $k^*=h^*$, then $k_2<k_1$  and $\k=k_2$. One can check that $T_{opt}$ verifies
\begin{equation*}
T_{opt}+\pi=2\pi k^*+(j+1)T_0+(k^*-1)(NT_0-2\pi)=2\pi+(Nk^*-j+1)T_0\,.
\end{equation*}
\end{itemize}
This completes the proof of $ii)$.\\

$iii)$ Since $N$ is odd, one has $N=2j+1$. Note that the observability sets which come from time less than $2\pi$ do not cover:
\begin{align}
\label{eq:F1}
F_h&:=((j+2)T_0-\pi + hT_0,(h+2)T_0)\,,\\
\label{eq:Fint}
F_{j+h}&:=((j+1+h)T_0,(h+1)T_0+\pi)\,,\\
\label{eq:F2}
F_{2j}&:=(NT_0,2\pi)\,.
\end{align}
where $h=0,2,\dots,j-2$.
We want to emulate the proof of $ii)$. We have only to pay attention to one issue: since $N$ is odd, if $\tilde{I}_k$ comes from an observation at $x=0$ (resp. $x=\pi$), then $\tilde{I}_{k+N+1}$ still comes from $x=0$ (resp. $x=\pi$).
The difference is that, $\tilde{I}_{k+N+1}$ is now ``shifted" towards the endpoint $t=2\pi$ by a quantity $(N+1)T_0-2\pi>0$ with respect to $\tilde{I}_{k}$.\\
By periodicity, the problem of finding $T_{opt}$ can be reduced to finding the minimum value of the time which allows us to cover $F_{j-2}, F_{2j-2}, F_{2j}$.\\
As in $ii)$, we split the proof into two steps. In the first one, we compute the least $\h\in \N$ such that $F_{j-1},F_{2j-2}$ and $F_{2j}$ are covered by subsets in $\{\tilde{I}_k\,:\,k\leq \h\}$. In the second one, $T_{opt}$ is found using Theorem \ref{general theorem}.\\
\\
\textbf{Step 1}.
We first analyse the covering of $F_{j-2}$. 
By the initial discussion, we can only consider the families of subsets $\I_1:=\{\tilde{I}_{j-2+n(N+1)}\,:\,n>0\}$ and $\I_2:=\{\tilde{I}_{N-2+n(N+1)}\,:\,n>0\}$. Each family provides an exhaustion of the missing set $F_{j-2}$ ``from the left to the right". Let us compute the least $h_i\in \N$ (with $i=1,2$) such that the subsets in $\{\tilde{I}_k\in \I_i\,:\,k\leq h_i\}$ cover $F_{j-2}$.\\
First consider the family $\I_1$. As usual, we ask for the minimum $n^*$ such that the following inclusion holds
\begin{multline*}
\bigcup_{n=0}^{n^*} (\tilde{I}_{j-2+n(N+1)}\cap F_{j-2}) =\\
 ((N-1)T_0-\pi,\min\{(j-1)T_0+ n^* [(N+1)T_0-2\pi],jT_0\})\supset F_{j-2}\,.
\end{multline*}  
By \eqref{eq:F1}, the previous inclusion holds if and only if $(j-1)T_0+ n^* [(N+1)T_0-2\pi]\geq jT_0$. Thus $n^*=l^*:=\lceil T_0/((N+1)T_0-2\pi)\rceil$ and
\begin{equation*}
h_1= (j-2)+l^*(N+1)\,.
\end{equation*}
Similar considerations hold for the family $\I_2$. One can readily check that
$$
h_2=N-2+q^*(N+1)\,,
$$
where 
\begin{equation}
\label{eq:qstarNdisparijpari}
q^*=\Big\lceil \frac{\pi-jT_0}{(N+1)T_0-2\pi}\Big\rceil = \Big\lceil \frac{T_0}{(N+1)T_0-2\pi}-\frac12\Big\rceil\,.
\end{equation}
All together, $\tilde{h}:=\min\{h_1,h_2\}$ is the least integer $m$ such that $\{\tilde{I}_k\,:\,k\leq m\}$ covers $F_{j-2}$. For $F_{2j-2}$ (resp. $F_{2j}$) a similar analysis applies with respect to the two families $\{\tilde{I}_{N-3+n(N+1)}\,:\,n\geq 1\}$, $\{\tilde{I}_{j-3+n(N+1)}\,:\,n\geq 1\}$ (resp. $\{\tilde{I}_{N-1+n(N+1)}\,:\,n\geq 1\}$, $\{\tilde{I}_{j-1+n(N+1)}\,:\,n\geq 1\}$). One can readily check that $\{\tilde{I}_k\,:\,k<\tilde{h}\}$ cover both $F_{2j-2}$ and $F_{2j}$. This implies $\h=\tilde{h}$.

\textbf{Step 2}. We compute $T_{opt}$. By \eqref{eq:qstarNdisparijpari} one has $q^*\leq l^*\leq q^*+1$, thus we have the following situations:
\begin{itemize}
\item If $l^*=q^*$, then $h_1< h_2$ and $\h=h_1$. One can check that
$$
T_{opt}=2\pi q^* + jT_0\,.
$$
\item If $q^*+1=l^*$, then $h_2<h_1$ and $\h=h_2$. Thus, $T_{opt}$ verifies
$$
T_{opt}-\pi=2\pi q^*+jT_0\,.
$$
\end{itemize}
This closes the proof of $iii)$.\\

$iv)$ (Sketch). Since $N$ is odd; then $N=2j+1$. Moreover, the sets $\tilde{I}_k$ with $k=0,\dots,N$, do not cover $[0,2\pi)$.  Namely the following subsets are not covered:
\begin{align*}
G_h&:=(T_0+hT_0,(j+2)T_0+hT_0-\pi)\,,&\quad & h=0,2,\dots,j-1	\,;\\
G_{j+1+h'}&:=(\pi+2T_0+h'T_0,(j+3)T_0+h'T_0)\,,&\,& h'=0,2,\dots,j-3\,.
\end{align*}
Note that $G_{j-1}=(jT_0,(2j+1)T_0-\pi)$ and is nonempty. As in $ii)$ we have assumed that $j\geq 3$, in the case $j=1$ the only subset to cover is $G_0:=(T_0,3T_0-\pi)$. 
Again, by a simple translation argument, we have only to study the covering of the sets $G_{j-1},G_{2j-2}$. Moreover, the least time $t$ which ensures that both sets are covered will provide the optimal control time $T_{opt}$. Lastly, $\pi+NT_0$ (resp.  $\pi+2jT_0$) is enough to cover $G_{j-1}$ (resp. $G_{2j-2}$). Since, $\pi+2jT_0<\pi+NT_0$, the claim is proven.
\end{proof}
\section{Dependence of $T_{opt}$ on $T_0$}
\label{s:dependence}
In the situation of Subsection \ref{ss:costantrate}, by Lemmas \ref{T0particolare}-\ref{T0generale}, for each $T_0 \in (0,2\pi)\setminus \{\frac{\pi}{2n+1}\,:\,n\in\0\}$, the optimal control time $T_{opt}$ is well defined. Thus, we may regard $T_{opt}$ as an application
\begin{equation}
  T_{opt}:(0,2\pi)\setminus \Big\{\frac{\pi}{2n+1}\,:\,n\in\0\Big\} \rightarrow \mathbb{R}\,,
  \qquad
T_0 \mapsto T_{opt}\,,
\end{equation}
and write it as $T_{opt}(T_0)$. Interestingly the formulas in Lemmas \ref{l:T0 large}-\ref{T0generale} permit us to study the behaviour of the map $T_{opt}(\cdot)$. Firstly, we consider the case $T_0\in (\pi,2\pi)$:

\begin{proposition}
The map $ T_{opt}(\cdot)$ on $(\pi,2\pi)$ admits a countable set of discontinuities in 
$$
\lambda_n :=\frac{n+2}{n+1}\pi\,, \qquad n\in \N\,.
$$
Moreover $\lim_{t\searrow \lambda_n} T_{opt}(t)=\pi(n+2)$ and $\lim_{t\nearrow \lambda_n} T_{opt}(t)=\pi(n+3)$.
\end{proposition}
\begin{proof}
The proof follows easily from Lemma \ref{l:T0 large}.
\end{proof}
Below, we will focus on the case $T_0\in (0,\pi)$. We start the analysis of the map focusing on the values of $T_0\in(0,\pi)$ for which $\lfloor \pi/T_0 \rfloor$ has a jump.
\begin{proposition}
\label{prop:discontinuita 1}
Let $T_0 = \pi/2n$ with $n\in \N$. Then $T_{opt}(T_0)=2\pi$, 
\begin{align}
\label{prop:lim 1 pari}
\lim_{t \nearrow T_0} T_{opt}(t)&=3\pi\,,\\
\label{prop:lim 2 pari}
\lim_{t \searrow T_0} T_{opt}(t)&=3\pi - \frac{\pi}{2n}\,.
\end{align} 
\end{proposition}
\begin{proof}
The first assertion is the content of $i)$ in Lemma \ref{T0particolare}. Instead in order to prove (\ref{prop:lim 1 pari}), we note that if $t$ belongs to a left neighborhood of $T_0$, we can write
\begin{equation*}
t = \frac{\pi}{2n+\varepsilon}\,,
\end{equation*}
for an appropriate $\varepsilon >0$; moreover $t\nearrow T_0$ if and only if $\varepsilon \searrow 0$.\\
Note that
\begin{equation*}
N(t)= 4n\,, \qquad j(t)=2n\,,
\end{equation*}
where $N$, $j$ are defined in (\ref{def:N,j}). Applying $i)$ in Lemma \ref{T0generale}, we obtain
\begin{equation*}
T_{opt}(t)=2\pi +2n\,\frac{\pi}{2n+\varepsilon}\,,
\end{equation*}
taking the limit for $\varepsilon \searrow 0$, one readily infers (\ref{prop:lim 1 pari}).\\
The limit in (\ref{prop:lim 2 pari}) is proved analogously. In fact let $t$ be in a small right neighborhood of $T_0$. So we can write $t$ as
\begin{equation*}
t = \frac{\pi}{2n-\varepsilon}\,,
\end{equation*}
for a small $\varepsilon >0$; as before $t\searrow T_0$ if and only if $\varepsilon \searrow 0$. By (\ref{def:N,j}) we have 
\begin{equation*}
N(t)= 4n-1\,, \qquad j(t)=2n-1\,.
\end{equation*}
Since $N$, $j$ are odd we can apply $iv)$ in Lemma \ref{T0generale}. So we get
\begin{equation*}
T_{opt}(t)=\pi + \frac{(4n-1)}{2n-\varepsilon}\pi \rightarrow \pi+\frac{(4n-1)}{2n}\pi  \,,
\end{equation*}
as $\varepsilon \searrow 0$.
\end{proof}
Lemma \ref{T0generale} and the argument used in the proof of Proposition \ref{prop:discontinuita 1} yield also the following results. The proof of Proposition \ref{prop:discontinuita 2} can be obtained easily by checking the behaviour of $k^*$ and $q^*$ in Lemma \ref{T0generale}.

\begin{proposition}
\label{prop:discontinuita 2}
Let $T_0=\pi/(2n+1)$ with $n \in \0$. Then
\begin{align}
\lim_{t\nearrow T_0} T_{opt}(t) = +\infty\,,\\
\lim_{t\searrow T_0} T_{opt}(t) = +\infty\,.
\end{align}
\end{proposition}

\begin{proposition}
\label{prop:discontinuita 3}
Let $T_0 = 2\pi / (2h +1)$ with $h \in \N$. Then\\
$i)$ If $h$ is even, 
\begin{align}
\lim_{t \nearrow T_0} T_{opt}(t)&=3\pi + \frac{2 h \pi}{2h +1}\,,\\
T_{opt}(T_0)=\lim_{t \searrow T_0} T_{opt}(t)&=2\pi + \frac{2h \pi}{2h +1}\,.
\end{align}
$ii)$ If $h$ is odd, we have
\begin{align}
T_{opt}(T_0)=\lim_{t \nearrow T_0} T_{opt}(t)&=3\pi\,,\\
\lim_{t \searrow T_0} T_{opt}(t)&=3\pi + \frac{2\pi h}{2h+1} \,.
\end{align}
\end{proposition}
Perhaps unexpectedly, there are other values of $T_0$ for which the map $T_0 \mapsto T_{opt}(T_0)$ is not continuous. For future convenience, we introduce the following notation
\begin{equation*}
  A_{k} =\left(\frac{\pi}{k+\frac{1}{2}}\,, \, \frac{\pi}{k} \right)\,,
  \qquad
  B_{k} = \left(\frac{\pi}{k+1}\,, \, \frac{\pi}{k+\frac{1}{2}} \right)\,,
  \qquad
  k\in\N
  \,.
\end{equation*}
Obviously 
\begin{equation}
\label{decomposition}
(0,\pi] = \Big[\cup_{k\in\N} (A_k \cup B_k) \Big]
\bigcup
\Big[\cup_{k\in\N}\Big\{\frac{\pi}{k}\,,\,\frac{\pi}{k+\frac{1}{2}} \Big\}\Big]\,.
\end{equation}
Note that the points in the set $\cup_{k\in\N}\Big\{\frac{\pi}{k}\,,\,\frac{\pi}{k+\frac{1}{2}} \Big\}$ have been considered in Propositions \ref{prop:discontinuita 1}-\ref{prop:discontinuita 3}. 
Next we look at the behaviour of the map $T_0\mapsto T_{opt}(T_0)$ in the sets $A_{k}$, $B_{k}$.
\begin{proposition}
\label{prop:discontinuita strane1}
For all $k\in 2\N$ (resp. $k \in 2\0+1$), the map $T_0\mapsto T_{opt}(T_0)$ is continuous on each interval $A_k$ (resp. $B_k$) .
\end{proposition}
The proof of Proposition \ref{prop:discontinuita strane1} follows straightforwardly from Lemma \ref{T0generale} when we keep in mind that: in $A_{k}$ we have $N=2k$, $j=k$, while in $B_{k}$ we have $N=2k+1$, $j=k$. Finally we explicitly list the discontinuities.
\begin{proposition}
\label{prop:discontinuita strane2}
If $k\in 2\0+1$, in each interval $A_k$, the map $T_{opt}(\cdot)$ admits a countable set of discontinuities in
\begin{equation}
\label{eq:discontinuityA}
\mu_m :=\frac{\pi}{k+\left(\frac{1}{m+2}\right)}\,, \qquad m\in \N\,.
\end{equation}
Furthermore, for each $m\in \N$, 
\begin{align*}
\lim_{t \nearrow \mu_{m}} T_{opt}(t)&=  \pi+ \mu_{m}(k(m+2)+1)\,,\\
\lim_{t \searrow \mu_{m}} T_{opt}(t)&= \pi+ \mu_{m}(k(m+3)+1)\,.
\end{align*}
In particular, $\lim_{t \searrow \mu_{m}} T_{opt}(t)-\lim_{t \nearrow \mu_{m}} T_{opt}(t)=k \mu_{m}$.
\end{proposition}

\begin{proposition}
\label{prop:discontinuita strane3}
For all $k\in 2\N$, the map $T_{opt}(\cdot)$ on each interval $B_k$ admits a countable set of discontinuities in
\begin{equation}
\label{eq:pointofdiscontinuity}
\xi_m:= \frac{\pi}{k+\left(\frac{m+1}{m+2}\right)}\,, \qquad m\in \N\,.
\end{equation}
Furthermore, for each $m\in \N$, 
\begin{align}
\label{eq:lim1}
\lim_{t \nearrow \xi_m} T_{opt}(t)&=  \pi(m+3) + k \xi_m\,,\\
\label{eq:lim2}
\lim_{t \searrow \xi_{m}} T_{opt}(t)&=  \pi(m+2) + k \xi_{m}\,.
\end{align}
In particular, $\lim_{t \nearrow \xi_{m}} T_{opt}(t)-\lim_{t \searrow \xi_{m}} T_{opt}(t)=\pi$.
\end{proposition}

We prove Proposition \ref{prop:discontinuita strane3}, the proof of Proposition \ref{prop:discontinuita strane2} being similar.

\begin{proof}[Proof of Proposition \ref{prop:discontinuita strane3}]
To begin, recall that $N(t)=2k+1$ and $j(t)=k$, for each $t\in B_k$. Thus $N$ is odd and $j$ is even and $T_{opt}$ is provided by $iii)$ in Lemma \ref{T0generale}. It is easy to see that the discontinuities of the map $B_k\ni t\mapsto T_{opt}(t)$ are the discontinuities of the maps
\begin{equation*}
q^*(t) = \Big\lceil  \frac{t}{(2k+2)t - 2\pi} -\frac{1}{2} \Big\rceil \,,\qquad
l^*(t) = \Big\lceil \frac{t}{(2k+2)t - 2\pi}  \Big\rceil \,,
\end{equation*}
where we have used \eqref{eq:qstarNdisparijpari}. Thus, the discontinuities of the map $t\mapsto T_{opt}(t)$ on $B_k$ are the solutions of the following equations
\begin{equation}
\label{eq:discontinuity}
\frac{t}{(2k+2)t-2\pi}=\frac{m}{2}+1\,, \quad m\in \N\,.
\end{equation}
With simple computations one obtains the sequence $\{\xi_m\}_{m\in\N}$ in \eqref{eq:pointofdiscontinuity}. It is easy to see that $\{\xi_m\}_{m\in\N}$ is a decreasing sequence. To conclude, we consider two cases:
\begin{itemize}
\item If $m\in 2\0$ and $t\in (\xi_{m+1},\xi_{m})$ (where we set $\xi_0:=\pi/(k+\frac{1}{2})$), one has 
$$
\frac{m}{2} +1<\frac{t}{2(k+1)t-2\pi}<\frac{m+1}{2}+1=\frac{m}{2}+\frac32\,.
$$
Thus $q^*=m/2+1$ and $l^*=m/2+2$. Lemma \ref{T0generale} $iii)$ yields
$$
T_{opt}(t)=2\pi \left(\frac{m}{2}+1\right) + jt +\pi =\pi(m+3)+jt\,.
$$
\item If $m\in 2\0+1$ and $t\in (\xi_{m+1},\xi_{m})$ we can argue as in the previous item. Set $m=2n+1$ for some $n\in\0$. Note that 
$$
n+\frac{3}{2} <\frac{t}{2(k+1)t-2\pi}<n+2\,.
$$
Thus $q^*=l^*=n+2$ and Lemma \ref{T0generale} $iii)$ yields
$$
T_{opt}(t)=2\pi(n+2)+ j t= \pi (m+3)+jt\,.
$$
\end{itemize}
Combining the two cases above, one easily concludes the proof.
\end{proof} 

\section{Multidimensional case}
\label{s:multidimensional proof}
In this section we prove Theorem \ref{t:variable theorem} and its extension to the $BV$ class. We use a result proved by the authors in \cite{ECC} (see Theorem 2.1 there); here we briefly recall it.\\
Let $N$ be an integer and
\begin{itemize}
\item let $\{x_i\}_{i=0,\dots,N}$ be an arbitrary family of points in $\mathbb{R}^d$;
\item let $\{t_i\}_{i=-1,\dots,N}$  be an increasing family of real numbers such that $t_{-1}=0$, and define $T:=t_N$;
\item define 
\begin{align}
\label{Ri}
R_i &=\max\{|x-x_i|\,|\,x\in \overline{\Omega}\}\,, &\qquad&  i=0,\dots,N\,,\\
\label{Ri+1,i}
R_{i,i+1} &= |x_{i+1}-x_i|\,,&\qquad&  i=0,\dots,N-1\,,\\
\Gamma_i &= \{x\in \partial \Omega \,|\, (x-x_i)\cdot \nu >0\}\,, &\qquad& i=0,\dots,N\,.
\end{align}
\end{itemize}
\begin{theorem}
\label{t:multid alternating}
For each real number $T$ such that
\begin{equation}
\label{T alternating}
T>R_N + \sum_{i=0}^{N-1} R_{i,i+1} + R_0\,,
\end{equation}
and each $u$ weak solution to (\ref{omo1})-(\ref{omo2}) with initial data $(u_0,u_1)\in H^1_0(\Omega)\times L^2(\Omega)$ the following observability inequality holds
\begin{equation}
\label{inequality multid alternating}
(\max_{i=0,\dots,N}R_i)\sum_{i=0}^{N}\int_{t_{i-1}}^{t_i}\int_{\Gamma_i} |\partial_{\nu} u|^2 d\mathcal{H}^{d-1} \,dt \geq C_T E_0\,,
\end{equation}
where $C_T =2(T-R_N - \sum_{i=0}^{N-1} R_{i,i+1} - R_0)$.
\end{theorem}

\begin{remark}
\label{r:alternating}
Theorem \ref{t:multid alternating} states that the system is $\Sigma$-observable (see Definition \ref{def:observabilitym}) in any time $T>R_N + \sum_{i=0}^{N} R_{i,i+1} + R_0$ with $\Sigma$ as 
\begin{equation}
\bigcup_{i=0}^{N} \Gamma_i\times (t_{i-1},t_i)\,.
\end{equation}
\end{remark}
In the proof of Theorem \ref{t:variable theorem}, we have to consider a sequence of partitions $\{P_k\,:\,k\in \N\}$ of the interval $[0,T]$, so we will add the upper index $k$ in order to keep trace of the dependence on $P_k$; and for brevity we will set $\Sigma_{\varphi}^k:= \Sigma_{\varphi}^{P_k}$.

With this preparation, we can prove Theorem \ref{t:variable theorem}.

\begin{proof}[Proof of Theorem \ref{t:variable theorem}]
Let $\{t_{-1}^k,\dots,t_{N_k}^k\,:\,0=t_{-1}^k<t_{0}^k<\dots t_{N_k-1}^k<t_ {N_k}^k=T\}$ be the partition $P_k$. We can apply Theorem \ref{t:multid alternating} for $x_i=x_{\varphi,i}^k:=\varphi(t_{i-1}^k)$ and $i=0,\dots,N_k$. Then the exact observability holds with $\Sigma = \Sigma_{\varphi}^k$ provided 
\begin{equation}
\label{T proof}
T > R_{N_k}^k + \sum_{i=0}^{N_k-1} R_{i+1,i}^k + R_0^k\,,
\end{equation}
where $R_i^k,R_{i+1,i}^k$ are as in (\ref{Ri})-(\ref{Ri+1,i}) with $x_i=x_{\varphi,i}^k$ (cf. Remark \ref{r:alternating}). 
We take a closer look at the middle term of the right hand side of  (\ref{T proof}):
\begin{equation}
\label{c1}
\sum_{i=0}^{N_{k}-1} R_{i+1,i}^k = \sum_{i=0}^{N_k} |\varphi(t_i^k) - \varphi(t_{i-1}^k)| - |\varphi(t_{N_{k}}^k) - \varphi(t_{{N_k}-1}^k)|\,.
\end{equation}
Since $\delta(P_k)\searrow 0$, invoking the smoothness hypotheses on $\varphi$, we see that the sum on the right hand side of (\ref{c1}) tends to the length of the curve
\begin{equation*}
L(\varphi):=\int_{0}^T |\varphi'(t)|\,dt\,.
\end{equation*} 
Furthermore, $R_{0}^k \equiv c_0$ since $t_{-1}^k\equiv 0$. By the continuity of $\varphi$ at $t=T$, we have $R_{N_{k}}^k \rightarrow c_T$ and $|\varphi(t_{N_{k}}^k) - \varphi(t_{{N_k}-1}^k)|\searrow 0$, as $k\nearrow \infty$ (recall that $|t_{N_{k}}^k-t_{N_{k}-1}^k|\leq \delta(P_k)\searrow 0$).\\ 
By the previous discussion and (\ref{Tvariable}), one sees that there exist an $M\in \N$ and a positive constant $\delta$ such that
\begin{equation}
\label{step5}
T - R_{N_k}^k - \sum_{i=0}^{N_k-1} R_{i+1,i}^k - R_0^k \geq \delta > 0\,,
\end{equation}
for all $k>M$. By inequality (\ref{inequality multid alternating}) in Theorem \ref{t:multid alternating}, we have
\begin{multline}
\label{estimatevariable}
(\max_{i=0,\dots,N}R_i^k)
\int_{(0,T)\times \partial \Omega} |\partial_{\nu} u|^2 \chi_{\Sigma_{\varphi}^k}\, d\mathcal{H}^{d-1}  \,\otimes \,dt =\\ (\max_{i=0,\dots,N}R_i^k)
\sum_{j=0}^{N_k} \int_{t_{j-1}^k}^{t_j^k} \int_{\Gamma_{\varphi,j}} |\partial_{\nu} u|^2\,d\mathcal{H}^{d-1}\,dt \geq 2\delta E_0\,,
\end{multline}
where $\delta$ is as in (\ref{step5}) and $\chi_A$ is the indicator function of the subset $A\subset \partial \Omega \times (0,T)$. It is easy to see that
\begin{equation}
\label{estimateRi}
\max_{t\in[0,T],x\in \overline{\Omega}}|x-\varphi(t)|\geq (\max_{i=0,\dots,N}R_i^k) \,.
\end{equation}
By (\ref{convergence}), we have at least for a subsequence
\begin{equation}
\label{convergence2}
\chi_{\Sigma_{\varphi}^k}(x,t) \rightarrow \chi_{\Sigma_{\varphi}}(x,t)\,,\qquad d\mathcal{H}^{d-1} \otimes dt - a.e.
\end{equation}
By the estimate in (\ref{estimateRi}), the uniformity of the estimate in (\ref{estimatevariable}), $\partial_{\nu} u \in L^2( \partial \Omega\times (0,T))$ (see Theorem \ref{t:hidden}) and the convergence in (\ref{convergence2}), the observability inequality (\ref{dis observabilitym}) follows by Lebesgue dominated convergence theorem.
\end{proof}

It is of some interest to weaken the hypotheses on the regularity of $\varphi$; see Remark \ref{r:BV} below.\\
For the reader's convenience, we recall that a function $\varphi:[0,T]\rightarrow\mathbb{R}^d$ is said to be of bounded variation (or briefly $\varphi\in BV[0,T]$) if
\begin{equation}
\label{BV def}
V_{0}^T(\varphi):= \sup \Big\{ \sum_{j=0}^N |\varphi(t_j)-\varphi(t_{j-1})| \Big\} <\infty\,,
\end{equation}
where the $\sup$ is taken on all partitions $P:=\{t_{-1},\dots,t_N\,:\,0=t_{-1}<t_0<\dots<t_{N-1}<t_N:=T\}$ of the interval $[0,T]$. 


\begin{theorem}
\label{t:variable BV}
Let $\varphi\in BV[0,T]$ be left continuous at $t=T$. Suppose that there exists a sequence of partitions $\{P_k\,:\,k\in \N\}$ such that $\delta(P_k)\searrow 0$ as $k\nearrow\infty$ and (\ref{convergence}) holds. Moreover, suppose that $T>0$ satisfies
\begin{equation}
\label{T BV}
T> c_0 + V_0^T(\varphi) + c_T\,.
\end{equation}
Then the system is $\Sigma_{\varphi}$-observable, where $\Sigma_{\varphi}$ is defined in (\ref{Sigma phi}).
\end{theorem}
\begin{proof}
The proof is an easy revision of the proof of Theorem \ref{t:variable theorem}. To avoid repetitions, we give only few remarks to stress the usefulness of the assumptions. To do so, we freely use the notation introduced in the proof of Theorem \ref{t:variable theorem} above.\\ 
The left continuity at $t=T$ is required to ensure that $R_{N^k}^k \rightarrow c_T$  as $k\nearrow\infty$. Moreover, since $\varphi\in BV[0,T]$, one has
\begin{equation}
\label{eq:BVproofTk}
T - R_{N_k}^k - \sum_{i=0}^{N_k-1} R_{i+1,i}^k - R_0^k \geq T- R_{N_k}^k- V_0^T(\varphi) -c_0  \,,
\end{equation}
since $\sum_{i=0}^{N_k-1} R_{i+1,i}^k= \sum_{i=0}^{N_k-1} |\varphi(t_{i-1}^k)-\varphi(t_{i}^k)|\leq V_0^T(\varphi)$ for all $k\in \N$ and $c_0\equiv R_0^k$ (recall that $t^k_{-1}\equiv 0$).
By \eqref{T BV}, \eqref{eq:BVproofTk} and $\lim_{k\rightarrow \infty}R_{N^k}^k =c_T$, then one sees that there exist $M\in \N$ and $\delta>0$ such that
$$
T - R_{N_k}^k - \sum_{i=0}^{N_k-1} R_{i+1,i}^k - R_0^k \geq \delta>0 \,, \quad \forall k>M\,.
$$
By the previous considerations, the conclusion follows as in \eqref{estimatevariable}-\eqref{convergence2}.
\end{proof}

We conclude this section with some remarks.

\begin{remark}
\label{r:BV}
Note that Theorem \ref{t:variable BV} is an extension of both Theorems \ref{t:variable theorem} and \ref{t:multid alternating}. Indeed, under the hypotheses of Theorem \ref{t:variable theorem} (thus $\varphi\in C([0,T];\mathbb{R}^d)\cap C^1((0,T);\mathbb{R}^d)$ with finite length), we have $\varphi\in BV[0,T]$ and $V_0^T(\varphi)=\int_0^T|\varphi'(t)|\,dt$. Thus the conclusion of Theorem \ref{t:variable theorem} follows by Theorem \ref{t:variable BV}.\\
Next we show that Theorem \ref{t:multid alternating} can be recovered from Theorem \ref{t:variable BV}. Let $P:=\{t_{-1},\dots,t_N\,:\,0=t_{-1}<t_0<\dots<t_{N-1}<t_N:=T\}$ be as in Theorem \ref{t:multid alternating}. Define the piecewise constant function $\tilde{\varphi}(t):=x_i$ on $ [t_{i-1},t_i)$ for $i=0,\dots,N$ and $\tilde{\varphi}(T)=x_N$. Then $\tilde{\varphi}$ is left continuous at $t=T$ and $V_0^T(\tilde{\varphi})=R_0+\sum_{i=0}^{N-1} R_{i,i+1}+R_N$. Lastly, set $P_k\equiv P$ for each $k\in \N$. Then, the condition \eqref{condition} is verified and the claim follows.

 For a discussion of the optimality of the lower bound in \eqref{T BV} see Example \ref{ex:optimality} below.
\end{remark}

\begin{remark}
Here we are not concerned with the optimal smoothness assumption on $\varphi$ in Theorem \ref{t:variable BV}. Our hypothesis $\varphi\in BV[0,T]$ is suggested by the lower bound in \eqref{T alternating}. Indeed, let $P=\{t_{-1},\dots,t_N\,:\,0=t_{-1}<t_0<\dots<t_{N-1}<t_N=T\}$ be a partition  of the interval $[0,T]$ and let $x_j=\varphi(t_{j-1})$ for $j=0,\dots,N$ (cf. \eqref{xphij}). Then, the middle term on the RHS of \eqref{T alternating} is naturally bounded by the total variation of the function $\varphi$ in the interval $[0,T]$. This property has been exploited in the proofs of Theorems \ref{t:variable theorem}, \ref{t:variable BV}. 
\end{remark}

\section{Applications and future directions}
\label{s:applications}
We look here at some applications of Theorems \ref{t:variable theorem} and \ref{t:variable BV}. In order to state the first result we fix some notations:
\begin{itemize}
\item let $B(0,1):=\{x\in \mathbb{R}^3\,:\,|x|<1\}$ be the unit ball with center in the origin, and set $S^2:=\partial B(0,1)=\{x\in \mathbb{R}^3\,:\,|x|=1\}$;
\item let $\Phi:[0,2\pi]\times [0,\pi] \rightarrow S^2$ be the standard parametrization of the sphere; i.e.
\begin{equation*}
\Phi(\theta,\theta')=(\cos\theta\,\sin\theta',\sin \theta\,\sin \theta', \cos \theta')\,;
\end{equation*}
\item let $\eta$ be a real number in $[0,2\pi]$, then define
\begin{equation*}
R(\eta) = \left( \begin{array}{ccc}
1\,&0\,&0\,\\
0\,&\cos \eta\,& -\sin \eta\,\\
0\,&\sin \eta\,& \cos \eta\,
\end{array}\right)\,.
\end{equation*}
\end{itemize} 
\begin{corollary}
\label{c:variable circle}
Let $\Omega=B(0,1)$ and let $\alpha\in \mathbb{R}^+$ be a positive real number, such that
\begin{equation}
\label{alpha bound}
\alpha < \frac{\pi}{2(1+\sqrt{2})+\pi \sqrt{2}}\,.
\end{equation}
Then the system is $\mathcal{S}$-observable in time $T=\pi/\alpha$ with $\mathcal{S}=\cup_{t\in(0,T)} S(t)\times \{t\}$ and
\begin{align}
\label{Gammat sphere}
S(t)&= R\left(\alpha\,t-\frac{\pi}{4}\right)S\,,\\ 
S&=\Big\{ \Phi(\theta,\theta')\subset S^2\,\Big|\, \theta \in \left(0,2\pi\right),\,\theta'\in \left(\frac{\pi}{4},\pi\right) \Big\}\,.
\end{align}
\begin{proof}
Consider 
\begin{equation}
\label{eq:phiapplication}
\varphi(t)=\sqrt{2}(0,\cos(\pi/4+\alpha t),\sin (\pi/4+\alpha t))\,, \quad t\in [0,\pi/\alpha]\,.
\end{equation} 
One may check that $\Gamma_{\varphi}(t)=S(t)$ for each $t\in (0,T)$, where the $S(t)$ is defined in (\ref{Gammat sphere}). By construction the length of $\varphi$ is equal to $\sqrt{2}\pi$ (note that it is independent of $\alpha$). Then Theorem \ref{t:variable theorem} ensures the $\mathcal{S}$-exact observability provided:
\begin{equation*}
\frac{\pi}{\alpha}=T> c_0 + L(\varphi)+c_T=2(1+\sqrt{2})+\sqrt{2}\pi\,;
\end{equation*}
where we have used that $c_0=c_T=1+\sqrt{2}$. Then the previous inequality is equivalent to \eqref{alpha bound}. This concludes the proof.
\end{proof}
\end{corollary}

\begin{remark}
We make some comments on Corollary \ref{c:variable circle} and we employ the notation of such corollary; in particular here $\varphi$ is as in \eqref{eq:phiapplication}.
\begin{itemize}
\item The condition \eqref{alpha bound} can be interpreted as the requirement that the mapping ``$t\mapsto \Gamma_{\varphi}(t)$" does not vary too fast in time. More specifically, $\alpha$ plays the role of the ``angular velocity" of rotation of the subset $S$; thus small $\alpha$ implies small velocity of rotation. We are not aware if this result is sharp, but this behaviour seems reasonable in view of geometrical optics; see for instance \cite{Labeau-Rauch,Variable}.
\item It is routine to prove that for each $\varepsilon>0$ there exists $\delta(\varepsilon)>0$ such that
\begin{equation}
\label{Gamma continuity}
\mathcal{H}^{d-1}(\Gamma_{\varphi}(t)\Delta \Gamma_{\varphi}(t')) < \varepsilon\,, \qquad\text{if}\,\, |t-t'|<\delta(\varepsilon)\,.
\end{equation}
The condition in (\ref{Gamma continuity}) can be viewed as a kind of continuity of the subset $\Gamma_{\varphi}(t)$ on the variable $t$, but this cannot be expected in cases when $\Omega$ is not smooth; see Example \ref{ex} below.
\end{itemize}
\end{remark}

The second item in the previous remark gives us the opportunity of pointing out the next observation. 

\begin{remark}
Let $\Omega$ and $\varphi$ be such that the continuity condition (\ref{Gamma continuity}) holds (e.g.  $B(0,1)\subset \mathbb{R}^3$ and $\varphi$ as in \eqref{eq:phiapplication}), then there exists a sequence of partitions $\{P_k\,:\,k\in \N\}$ for which the condition (\ref{convergence}) holds. To prove the claim, let $\varepsilon=1/k$ and let $\delta_k>0$ be such that the condition \eqref{Gamma continuity} holds true. Let $ \N\ni \K>1/\delta_k$, define $t_j^k:=(j+1)T/\K$ for $j=-1,\dots,\K-1$ and $P_k:=\{t_{-1}^k,\dots,t_{\K-1}^k\}$. Then, $\delta(P_k)\searrow 0$ and
\begin{align*}
&\lim_{k\rightarrow\infty}(\mathcal{H}^{d-1}\otimes \mathcal{L}^1)(\Sigma_{\varphi}\Delta \Sigma_{\varphi}^{P_k})
 \\&\leq \lim_{k\rightarrow\infty}
\sum_{i=0}^{K-1} (t_i^k-t_{i-1}^k) \sup_{t\in (t_{i-1}^k,t_i^k)}\mathcal{H}^{d-1}(\Gamma_{\varphi}(t)\Delta \Gamma_{\varphi}(t_{i-1}^k))
 \leq
\lim_{k\rightarrow\infty}\frac{1}{k} T=0\,.
\end{align*}
Therefore, in Theorem \ref{t:variable BV}, assumption \eqref{Gamma continuity} can replace \eqref{convergence}.
\end{remark}

In the next example we study in some details an application in which \eqref{Gamma continuity} does not hold.

\begin{example}
\label{ex}
Let $\Omega=(0,1)^2$ be the square with unit side length and let $\beta$ be a positive number; then consider the curve
\begin{equation}
\label{varphi square}	
\varphi(t)=\left(2, - \frac{1}{2} + \beta t\right)\,, \qquad t\in \left[0,\frac{1}{\beta}\right]\,.
\end{equation}
By definition (see (\ref{Sigma phi})) the observability set $\Sigma_{\varphi}\subset \partial (0,1)^2\times (0,1/\beta)$ can be written as
\begin{equation}
\label{sigma square}
\Sigma_{\varphi}= \left(l_0 \times\left(0,\frac{1}{2\beta} \right] \right)\cup \left( l_1 \times \left(\frac{1}{2\beta},\frac{1}{\beta}\right)\right)\,,
\end{equation}
where
\begin{align*}
l_0 &= \{(x,y)\,|\, x=0\, \text{and}\,\, y \in (0,1)\,\,, \text{or}\,\, y=1\,\text{and}\,\, x\in (0,1)\}\,,\\
l_1 &= \{(x,y)\,|\, x=0\, \text{and}\,\, y \in (0,1)\,\,, \text{or}\,\, y\in \{0,1\}\,\text{and}\,\, x\in (0,1)\}\,.
\end{align*}

Furthermore, for any $t,t'\in (0,\pi)$ such that $t>1/2>t'$ we have
\begin{equation*}
\mathcal{H}^{d-1} (\Gamma_{\varphi}(t)\Delta \Gamma_{\varphi}(t')) = \mathcal{H}^{d-1}(\{(x,y)\,|\, y=1\,,x\in(0,1)\})=1\,;
\end{equation*}
in this case we lose the continuity condition in (\ref{Gamma continuity}).
\end{example}

Next, we note that Theorem \ref{t:multid alternating} (or Theorem \ref{t:variable BV}) implies the following result (for the case of a circle see \cite{ECC} Section 4).
\begin{corollary}[$N$-times alternating control]
\label{c:square alternating N}
Let $T$ and $N$ be respectively a positive real and integer number, such that
\begin{equation}
\label{T square alternating}
T>\sqrt{2}(N+2)\,,
\end{equation}
and let $\{t_i\}_{i=0,\dots,N-1}$ be any increasing finite sequence of $(0,T)$. Then the system is $ \Sigma$-observable for $\Omega=(0,1)^2$ and
\begin{equation*}
\Sigma =\left( \bigcup_{i=0,\,i\in 2\mathbf{N}}^N d_0 \times (t_{j-1},t_j)\right)\cup \left( \bigcup_{i=0,\,i\in 2\mathbf{N}+1}^N d_1 \times (t_{j-1},t_j)\right)\,,
\end{equation*}
where
\begin{align*}
d_0 &=\Big\{(x,y)\in \partial \Omega \,\Big|\,x=0\,\,\text{and}\,\,y\in(0,1)\,,\,\, \text{or}\,\,y=0\,\,\text{and}\,\,x\in(0,1) \Big\},\\
d_1 &=\Big\{(x,y)\in \partial \Omega \,\Big|\,x=1\,\,\text{and}\,\,y\in(0,1)\,,\,\, \text{or}\,\,y=1\,\,\text{and}\,\,x\in(0,1) \Big\},\\
\end{align*}
\end{corollary}
\begin{proof}
Apply Theorem \ref{t:multid alternating} with $x_i\equiv x_0=(1,1)$ if $i$ is even or $x_i\equiv x_1=(0,0)$ otherwise.
\end{proof}
For a general discussion on the choice of the $\{t_i\}_{i=0,\dots,N-1}$ see \cite[Subsection 2.1]{ECC}. Let us specialize Corollary \ref{c:square alternating N} to the case 
\begin{equation}
\label{tk alternating}
t_{k}=(k+1)T_0\,,\qquad k=-1,\dots,N\,,
\end{equation}
thus $T:=t_N =(N+1)T_0$. From (\ref{T square alternating}) we obtain
\begin{equation*}
(N+1) T_0 > \sqrt{2}(N+2)\,,
\end{equation*}
that leads us to a lower bound for $T_0$
\begin{equation}
\label{T lower bound}
T_0 >\sqrt{2}\,\, \frac{N+2}{N+1}>\sqrt{2}\,.
\end{equation}
In the situation just introduced, we note that:
\begin{itemize}
\item In the fixed case for the square (e.g. the control acts only on the subset $d_0$ for each time) the controllability holds for any $T>2\sqrt{2}$, by this for $\sqrt{2}<T_0\leq 2\sqrt{2}$ Corollary \ref{c:square alternating N} cannot be deduced by the fixed controllability case.
\item The lower bound in (\ref{T lower bound}) is inherent in the application of the multiplier method in the proof of the observability inequality. In the one dimensional case instead we can allow for arbitrarily small $T_0$ (see Subsection \ref{ss:costantrate}). The one dimensional case suggests that the limitation in (\ref{T lower bound}) is due to methodology of proof and not to an intrinsic obstruction of the problem.\\
It is of our interest to know if there exists a sequence $\{T_{0,h}\}_{h\in \N}$, such that $T_{0,h}\searrow 0$ as $h\nearrow \infty$ and for all $h\in \N$ the system of Corollary \ref{c:square alternating N}  is exactly observable.
\end{itemize}
Up to now we have not discussed the optimality of the lower bound in (\ref{Tvariable}) provided in Theorem \ref{t:variable theorem}; or more generally the optimality of (\ref{T BV}) in Theorem \ref{t:variable BV}. The next example shows that the lower bound for $T$ in (\ref{T BV}) is not optimal in general.
\begin{example}
\label{ex:optimality}
Consider $\Omega=(0,1)^2$, take any $\varphi:[0,T]\rightarrow \mathbb{R}^2$ continuous and piecewise differentiable of \textit{positive length}, such that $\varphi((0,T))\subset \Omega$ and $\varphi(0)=\varphi(T)=0$. From the convexity of $\Omega$, it follows that
\begin{equation*}
\Gamma_{\varphi}(t) \equiv \partial \Omega\,, \qquad \forall t\in (0,T)\,.
\end{equation*}
So (\ref{T BV}), or equivalently (\ref{Tvariable}), provides 
\begin{equation}
\label{square non optimal}
T> 2\sqrt{2} + L(\varphi)\,.
\end{equation}
It is known (see \cite{multiplier} or \cite{Lions}) that the system is exactly observable for any $T>\sqrt{2}$, so the inequality in (\ref{square non optimal}) is not optimal.
\end{example}
Besides this, it is unknown if the lower bound \eqref{T BV} in Theorem \ref{t:variable BV} is optimal for specific domain $\Omega$ and observability set $\Sigma$; e.g. the ones given in Corollary \ref{c:variable circle}.

\section{Conclusion}
In this article we extend some classical results concerning the exact boundary observability of the wave equation.\\ 
In the one dimensional case, we provide a complete treatment of the observability problem. Indeed, Theorem \ref{general theorem} and Corollary \ref{corollaryJ} completely characterize the boundary observability in terms on how the subsets of observation cover the one dimensional sphere $S^1$; see the condition in (\ref{general condition}) or in (\ref{c:eq corollaryJ}).\\
In the multidimensional case, we work out sufficient conditions for exact boundary observability; see Theorem \ref{t:variable theorem} or more generally Theorem \ref{t:variable BV}. The proofs of these results are based on the multiplier method (see \cite{Lions} or \cite{multiplier}), which is used to prove the observability in case of constant observation, i.e. when the observation subset is $ \Sigma \equiv \Gamma \times (0,T)$ for some $\Gamma\subset \partial \Omega \times (0,T)$. Our results extends the classical ones even in the case of piecewise constant subset of observation (see Remark \ref{r:BV}).\\
The applicability of our results is exemplified in Section \ref{s:applications}, where we also list some perspective developments. We recall here the issues of alternating control with arbitrarily small switch time $T_0$, and the optimality of the controllability time $T$, both in the multidimensional case.

\paragraph{\textbf{Acknowledgement}} We are indebted to the anonymous referees for their detailed and very helpful comments, which led to significant corrections, pointing out mistakes in the original proof of Lemma \ref{T0generale}.

\end{document}